%% file: AtBe25.tex
\documentclass[numbers,webpdf,imaiai]{ima-authoring-template}%

\usepackage[T1]{fontenc}
\usepackage[utf8]{inputenc}
\usepackage{algorithm}
\usepackage{amsmath, mathtools,xfrac}
\usepackage{enumerate}
\usepackage{tikz,lipsum,lmodern}
\usepackage{graphicx}
\usepackage{xcolor}
\usepackage{lmodern}

\usepackage[capitalise,noabbrev,capitalize]{cleveref}

\newtheorem{assumption}{Assumption}
\crefname{assumption}{Assumption}{Assumptions}
\Crefname{assumption}{Assumption}{Assumptions}

\usepackage[american]{babel}
\usepackage[T1]{fontenc}
\usepackage[utf8]{inputenc}
\usepackage{algpseudocode}
\usepackage{xtemplate}
\usepackage{pgfplots}
\usepackage[most]{tcolorbox}
\usepackage[normalem]{ulem}

\newcommand{\stkout}[1]{\textcolor{red}{\ifmmode\text{\sout{\ensuremath{#1}}}\else\sout{#1}\fi}}

\usepackage{lmodern}
\DeclareSymbolFont{largesymbols}{OMX}{cmex}{m}{n} %

\usepackage[textsize=tiny]{todonotes}

\graphicspath{{Fig/}}
\newcommand{\numsample}{m}
\newcommand{\numpixel}{N}
\newcommand{\A}{\mathbf{A}}

\newcommand{\mphi}{\mathbf{\Phi}}

\newcommand{\DictionaryMat}{{\boldsymbol{\Phi}}}

\newcommand{\G}{\mathbf{G}}

\newcommand{\DictionaryMatOne}{\mphi_1}
\newcommand{\DictionaryMatTwo}{\mphi_2}
\newcommand{\DictionaryPsiOne}{\mathbf{\Psi}_1}
\newcommand{\DictionaryPsiTwo}{\mathbf{\Psi}_2}

\DeclareMathOperator{\Real}{Re}

\DeclareMathOperator{\identityMapping}{id}
\DeclareMathOperator{\prox}{prox}

\DeclareMathOperator*{\argmin}{arg\,min}
\providecommand{\abs}[1]{\left\lvert#1\right\rvert}
\providecommand{\Abs}[1]{\left\lVert#1\right\rVert}
\providecommand{\brac}[1]{\left(#1\right)}

\providecommand{\scp}[2]{\left\langle#1\,\middle\vert\,#2\right\rangle}

\theoremstyle{thmstyletwo}%
\newtheorem{theorem}{Theorem}%
\newtheorem{proposition}{Proposition}%

\newtheorem{remark}{Remark}%
\newtheorem{definition}{Definition}
\newtheorem{corollary}{Corollary}
\newtheorem{lemma}{Lemma}

\numberwithin{equation}{section}

\definecolor{o}{RGB}{255,127,14}
\definecolor{b}{RGB}{31,119,180}

\begin{document}

\DOI{DOI HERE}
\copyrightyear{2025}
\vol{00}
\pubyear{2021}
\access{Advance Access Publication Date: Day Month Year}
\appnotes{Paper}
\copyrightstatement{}
\firstpage{1}

\title[A generalization bound for exit wave reconstruction via deep unfolding]{A generalization bound for\\exit wave reconstruction via deep unfolding}

\author{Moussa Atwi* and Benjamin Berkels\ORCID{0000-0002-6969-187X}
\address{\orgdiv{IGPM}, \orgname{RWTH Aachen University}, \orgaddress{\street{Schinkelstr. 2}, \postcode{52062 Aachen}, \country{Germany}}}}

\authormark{Moussa Atwi and Benjamin Berkels}

\corresp[*]{Corresponding author: \href{email:atwi@ssd.rwth-aachen.de}{atwi@ssd.rwth-aachen.de}}

\received{Date}{0}{Year}
\revised{Date}{0}{Year}
\accepted{Date}{0}{Year}

\abstract{Transmission Electron Microscopy enables high-resolution imaging of materials, but the resulting images are difficult to interpret directly. 
One way to address this is exit wave reconstruction, i.e., the recovery of the complex-valued electron wave at the specimen's exit plane from intensity-only measurements. This is an inverse problem with a nonlinear forward model.
We consider a simplified forward model, making the problem equivalent to phase retrieval, and propose a discretized regularized variational formulation. To solve the resulting non-convex problem, we employ the proximal gradient algorithm (PGA) and unfold its iterations into a neural network, where each layer corresponds to one PGA step with learnable parameters. This unrolling approach, inspired by LISTA, enables improved reconstruction quality, interpretability, and implicit dictionary learning from data. We analyze the effect of parameter perturbations and show that they can accumulate exponentially with the number of layers $L$. Building on proof techniques of Behboodi et al., originally developed for LISTA, i.e., for a linear forward model, we extend the analysis to our nonlinear setting and establish generalization error bounds of order $\mathcal{O}(\sqrt{L})$. Numerical experiments support the exponential growth of parameter perturbations.
}

\keywords{Phase retrieval; TEM imaging; sparse reconstruction; deep unfolding; proximal gradient algorithm; generalization error; Rademacher complexity.}

\maketitle

\section{Introduction}

Transmission Electron Microscopy (TEM) is a microscopy technique in which two-dimensional images are created by recording a beam of electrons passing through a thin sample (specimen). After the electrons interact with the material and leave the exit plane of the specimen, they go through a system of electromagnetic lenses and finally their squared amplitude is recorded, forming a TEM image.
However, the direct interpretation of TEM images is difficult because they are blurred by aberrations from the lenses and further affected by the loss of phase information. One possible way to deal with these limitations is to reconstruct the electron wave at the exit plane of the specimen, which is known as the \emph{exit wave} \cite{CoThOp96}. Exit wave reconstruction provides access to phase information that is otherwise lost in direct TEM imaging.
This results in a nonlinear inverse problem that can be seen as a generalized variant of phase retrieval.
A range of strategies has been developed for this task, from early focus-variation approaches addressing both linear and nonlinear imaging effects \cite{CoThOp96} to exit wave reconstruction using the transport of intensity equation combined with self-consistent propagation \cite{HsChKa04}. More recent work introduced a variational framework for exit wave reconstruction \cite{DoBe19} generalizing some of the older approaches, others proposed improved focal-series algorithms that incorporate noise suppression and allow tracking structural evolution in nanomaterials \cite{ZhChWa24}.

The variational approach from \cite{DoBe19} forms the starting point for us.
We consider a simplified forward model where the nonlinearity is just element-wise phase loss.
Despite the simplification, this is still a challenging nonlinear inverse problem.
To address it, we use the proximal gradient algorithm (PGA) to solve the regularized variational problem. Building on advances in algorithm unrolling, we unfold the iterative PGA into a neural network architecture inspired by the Iterative Shrinkage-Thresholding Algorithm (ISTA) and its learned counterpart LISTA. This unfolding maps each iteration to a network layer with learnable parameters, providing a structured and interpretable model that combines traditional optimization with deep learning \cite{GrLe10,MoLiEl21}.

A critical aspect here is understanding the generalization ability of the unfolded network, i.e., how well a learned neural network performs on unseen data. Generalization bounds provide probabilistic estimates of the gap between the true error with respect to the unknown distribution and the empirical error on the training set, and thus form the theoretical basis of our analysis. Using tools from statistical learning theory \cite{BeRaSc22,Ta21}, particularly building on the strategy to derive generalization bounds for LISTA from \cite{BeRaSc22}, we derive such a bound for our exit wave reconstruction variant.
The bound has the form%
\[
\mathcal{L}(h) - \hat{\mathcal{L}}(h) \;\lesssim\; \sqrt{\frac{N^2}{m}} (1+\sqrt{K})\sqrt{L},
\]
where $L$ is the depth of the network, $N^2$ the number if trainable parameters, $KN^2$ the number of measurements and $m$ the number of samples.

\section{Forward Model and Variational Approach}

As described in the introduction, exit wave reconstruction aims to recover the complex-valued exit wave from intensity-only measurements recorded at different defocus levels. The following presents both the full forward model and a simplified version suitable for deep unfolding, leading to a variational approach.

\subsection{A Simplified Exit Wave Model - Phase Retrieval}

\paragraph{Full forward model:} Let $\Psi \in L^2(\mathbb{R}^d;\mathbb{C})$ denote the complex-valued exit wave, $T_j \in L^\infty(\mathbb{R}^d \times \mathbb{R}^d;\mathbb{C})$ the transmission cross-coefficient (TCC). We assume a sequence of real-valued TEM images $(\tilde{g}_j)_{j=1}^K$ has been recorded at different defocus levels. Our objective is to reconstruct $\Psi$ from these intensity-only measurements. This means to find $\Psi$ such that
\[
\Psi \star_{T_j} \Psi = \mathcal{F}(\tilde{g}_j)\text{ for }j=1,\ldots K.
\]
Here, \(\Psi \star_{T_j} \Psi\) is the weighted auto-correlation, i.e.,
\[
(\Psi \star_{T_j} \Psi)(x) = \int_{\mathbb{R}^d} \Psi(x+y)\, T_j(x+y, y)\, \Psi^*(y) \, \mathrm{d}y.
\]
This gives rise to the objective functional
\[
\mathcal{E}[\Psi] = \frac{1}{K} \sum_{j=1}^K \left\| \Psi \star_{T_j} \Psi - \mathcal{F}(\tilde{g}_j) \right\|_{L^2}^2 + \mathcal{R}[\Psi],
\]
where $\mathcal{R}[\Psi]$ is a regularization term.  
\paragraph{Simplified forward model:} To reduce computational complexity, we assume the TCCs are separable, i.e., that there exist weighting functions $w_j \in L^\infty(\mathbb{R}^d;\mathbb{C})$ such that
$
T_j(x, y) = \overline{w_j(x)} w_j(y).
$
Under this assumption, the weighted auto-correlation simplifies to
$
\Psi \star_{T_j} \Psi = (\Psi w_j) \star (\Psi w_j),
$
and using the convolution theorem
$
\mathcal{F}^{-1}(\Psi \star_{T_j} \Psi) = \left| \mathcal{F}^{-1}(\Psi w_j) \right|^2,
$
we get the simplified forward model
\[
\tilde{g}_j = \left| \mathcal{F}^{-1}(\Psi w_j) \right|^2
\]
and the simplified objective functional
\[
\mathcal{E}[\Psi] = \frac{1}{K} \sum_{j=1}^K \left\| \left| \mathcal{F}^{-1}(\Psi w_j) \right|^2 - \tilde{g}_j \right\|_{L^2}^2 + \mathcal{R}[\Psi].
\]
Despite being simplified compared to the original exit wave reconstruction problem, this is still an inverse problem with a nonlinear forward model.

\subsection{Discretization}

To solve the variational problem numerically, we discretize the domain using $\numpixel$ spatial nodes. The continuous exit wave becomes a vector $\pmb{\psi} \in \mathbb{C}^\numpixel$, and each image $\tilde{g}_j$ becomes a vector $\tilde{\mathbf{g}}_j \in \mathbb{R}^\numpixel$.

We define measurement operators $A_j \in \mathbb{C}^{\numpixel \times \numpixel}$ representing an inverse discrete Fourier transform followed by pointwise multiplication with the weighting function
\[
A_j\pmb{\psi} = \mathcal{F}_d^{-1}(\pmb{\psi} \odot w_j),
\]
where \(\odot\) denotes element-wise multiplication.
The discretized objective function becomes:
\[
\tilde{\mathcal{E}}(\pmb{\psi}) = \tilde{\mathcal{D}}(\pmb{\psi}) + \mathcal{R}(\pmb{\psi})
\text{, with data term }
\tilde{\mathcal{D}}(\pmb{\psi}) := \frac{1}{2K\numpixel} \sum_{j=1}^K \left\| \tilde{\varphi}(A_j \pmb{\psi}) - \tilde{\mathbf{g}}_j \right\|_2^2.
\]
Here, $\tilde{\varphi} : \mathbb{C} \to \mathbb{R}, \,\tilde{\varphi}(z) = |z|^2$, is applied element-wise, i.e.,
\[
\tilde{\varphi}(\mathbf{z}) := \left( \tilde{\varphi}(z_1), \dots, \tilde{\varphi}(z_{K\numpixel}) \right), \quad \mathbf{z} \in \mathbb{C}^{K\numpixel}.
\]
By stacking the measurement operators and image vectors, i.e.,
\begin{equation}\label{ktimes}
A = \begin{pmatrix}
A_1 \\ \vdots \\ A_K
\end{pmatrix} \in \mathbb{C}^{K\numpixel \times \numpixel}, \quad
\tilde{\mathbf{g}} = \begin{pmatrix}
\tilde{\mathbf{g}}_1 \\ \vdots \\ \tilde{\mathbf{g}}_K
\end{pmatrix} \in \mathbb{R}^{K\numpixel},
\end{equation}
the data term can be written compactly as
\[
\tilde{\mathcal{D}}(\pmb{\psi}) := \frac{1}{2K\numpixel} \left\| \tilde{\varphi}(A \pmb{\psi}) - \tilde{\mathbf{g}} \right\|_2^2.
\]

\subsection{Variational Perspective}

We introduce a general transformation strategy to improve robustness. Let $\gamma \in C^1([0, \infty), [0, \infty))$ be a strictly increasing function, and define
\[
\varphi := \gamma \circ \tilde{\varphi}, \quad \mathbf{g} := \gamma(\tilde{\mathbf{g}}),
\]
so that
\[
\varphi(\mathbf{A} \pmb{\psi}) = \mathbf{g}.
\]
A possible choice is $\gamma(x) = \sqrt{\abs{x} + \delta^2} - \delta, \quad \delta > 0$. In this case,
\[
\varphi(x) = \sqrt{\abs{x}^2 + \delta^2} - \delta, \quad \delta > 0.
\]
i.e., $\varphi$ is the scaled Pseudo Huber transformation. It has been observed that the smoothing of the absolute value done here not only improves the theoretical properties of the objective function, but can also improve the practical performance of reconstruction algorithms \cite{XiWaGu21}.

The resulting smoothed energy functional is:
We write the energy as:
\begin{equation}
\label{eq:EWObjective}
\mathcal{E}[\pmb{\psi}] = \mathcal{D}(\pmb{\psi}) + \mathcal{R}(\pmb{\psi}),
\end{equation}
where the data term is
\[
\mathcal{D}(\pmb{\psi}) := \frac{1}{2KN} \|\varphi(\mathbf{A} \pmb{\psi}) - \mathbf{g}\|_2^2,
\]
and $\mathcal{R}(\pmb{\psi})$ is a convex regularizer, e.g., promoting sparsity via $\ell_1$-norm.

\paragraph{Optimization via Proximal Gradient Method}

The minimization is performed using the proximal gradient method, adapted to complex variables via the Wirtinger derivative:
\[
\pmb{\psi}^{k+1} = \operatorname{prox}_{\tau_k \mathcal{R}} \left( \pmb{\psi}^k - \tau_k \nabla \mathcal{D}(\pmb{\psi}^k) \right),
\]
with gradient
\[
\nabla \mathcal{D}(\pmb{\psi}) = \frac{1}{KN} \overline{\mathbf{A}^\intercal} \left( \varphi(\mathbf{A} \pmb{\psi}) - \mathbf{g} \right) \odot \varphi'(\mathbf{A} \pmb{\psi}).
\]
Since the $\mathcal{D}$ is real-valued, its Wirtinger derivative coincides with the standard (real) derivative when interpreting $\mathbb{C}$ as $\mathbb{R}^2$.

\subsection{Deep Unfolding}

Deep unfolding, or algorithm unrolling, is a technique that transforms classical iterative algorithms into trainable neural networks and introduces parameters that can be learned from data, resulting in a hybrid model that leverages both domain knowledge and learning capabilities \cite{MoLiEl21}.

Traditional deep networks, while powerful, often operate as black boxes with complex structures that make their decision-making difficult to interpret. Their performance heavily relies on large, high-quality training datasets, which are often expensive to acquire, especially in fields such as medical imaging or electron microscopy. In contrast, classical iterative algorithms are transparent and theoretically grounded, offering a clear link between algorithm steps and the problem structure.

By unrolling such algorithms into neural networks, we retain interpretability and domain fidelity, while benefiting from the hybrid model's faster convergence and parameter efficiency. Here, the convergence speed improvement is relative to the baseline iterative method, and the reduced parameter count arises from the structured, hybrid design rather than from deep learning alone. Each layer of the network mimics one iteration, and the resulting model can be trained end-to-end on real data. This allows the network to generalize well even with limited data and provides a more explainable and data-efficient alternative to standard deep architectures.

The foundation of deep unfolding was notably laid by the unrolling of ISTA (Iterative Shrinkage-Thresholding Algorithm), originally formulated in \cite{DaDeMo04}, and later unrolled by Gregor and LeCun in their introduction of Learned ISTA (LISTA) \cite{GrLe10}. In their experimental study, they demonstrated that LISTA significantly outperformed the classical ISTA algorithm in terms of reconstruction speed and accuracy. This substantial improvement in performance was primarily attributed to the unique network architecture enabled by the unrolling process itself, which reinterprets the optimization dynamics as a trainable feed-forward structure. Another notable example of algorithm unrolling is in blind image deblurring, where the iterative steps of a Half-Quadratic Splitting algorithm for gradient-domain deblurring are unrolled to construct a deep network, referred to as DUBLID \cite{LiToMo19}. In this approach, each iteration corresponds to a network layer, and key parameters such as filters and thresholds are learned from data, resulting in both interpretable and computationally efficient deblurring. Similarly, the Deep Griffin--Lim Iteration (DeGLI) framework for phase reconstruction \cite{MaYaKo21} unrolls the classical Griffin-Lim algorithm, with each iteration augmented by a trainable residual DNN block. %
Another significant contribution to the field was made by Hershey et al. \cite{HeRoWe14}, who formalized a general deep unfolding framework. They demonstrated how iterative model-based inference algorithms, such as belief propagation and non-negative matrix factorization, can be unrolled into trainable deep network architectures. %

\subsection{Unrolling the Proximal Gradient Algorithm into a Neural Network}\label{sec: unrolling}

We now apply deep unfolding to the proximal gradient algorithm applied for the minimization of \eqref{eq:EWObjective}, i.e., for variational exit wave reconstruction as introduced above.

We model each exit wave $\pmb{\psi} \in \mathbb{C}^N$ as a sparse signal with respect to an (unknown) dictionary in the form of a unitary matrix $\mphi_0\in \mathcal{U}(N)$, i.e., $\pmb{\psi} = \mphi_0 \mathbf{z}$ for some sparse vector $\mathbf{z} \in \mathbb{C}^N$. Here, $\mathcal{U}(N) := \left\{ \mathbf{U} \in \mathbb{C}^{N \times N} : \overline{\mathbf{U}}^\top \mathbf{U} = \mathbf{I} \right\}$ is the set of unitary $N \times N$ matrices.
We assume that this sparse generative model holds across the training data, i.e., for training samples  $\big((\pmb{\psi}_j, \mathbf{g}_j)\big)_{j=1}^m$, each fulfills $\pmb{\psi}_j = \mphi_0 \mathbf{z}_j$ and $\mathbf{g}_j = \varphi(\mathbf{A} \pmb{\psi}_j)$.
The corresponding loss function measures how well the reconstruction of $\pmb{\psi}_j$ from the estimated sparse code and dictionary matches the true exit wave.
To obtain the sparse codes $\mathbf{z}_j$ for one exit wave, only the measurements $\mathbf{g}_j$ are used. In contrast, the dictionary $\mphi$ is learned from the full training set  $\big((\pmb{\psi}_j, \mathbf{g}_j)\big)_{j=1}^m$.

We design an unrolled neural network architecture that mimics $L$ iterations of the proximal gradient algorithm, replacing the unknown fixed dictionary $\mphi_0$ by a trainable parameter $\mphi\in\mathcal{U}(N)$. Each iteration is referred to as a \emph{stage}, consisting of a gradient descent step and an application of the proximal operator. Each stage maps an input sparse code estimate to a refined estimate, guided by the observed intensity data $\mathbf{g} \in \mathbb{R}^{KN}$ and the learned dictionary $\mphi$.

We assume that signals $\pmb{\psi}$ are bounded in the $\ell_2$-norm by a fixed constant $C_{\text{in}}$, i.e., $\lVert \pmb{\psi} \rVert_2 \leq C_{\text{in}}$. Note that this constant will appear in the generalization bound we are going to derive.

The $\ell$-th stage of the network is defined as:
\begin{equation}
f_{\ell+1}^{\mathbf{g}, \mphi}(\mathbf{z}) := \operatorname{prox}_{\tau_\ell \mathcal{R}} \left( \mathbf{z} - \frac{\tau_\ell}{KN} (\overline{\mathbf{A} \mphi})^\intercal \left( \varphi(\mathbf{A} \mphi \mathbf{z}) - \mathbf{g} \right) \odot \varphi'(\mathbf{A} \mphi \mathbf{z}) \right), \label{eq:unroll_stage}
\end{equation}
Here, by replacing the original signal $\pmb{\psi}$ with a different basis representation $\mphi \mathbf{z}$, we are effectively replacing $A$ by $A \mphi$ in the data term $\mathcal{D}$. The matrix $\mphi \in \mathbb{C}^{N \times N}$ is searched over the space of unitary matrices to best fit the observed measurements.  
The step sizes $\tau_\ell$ are fixed and not learned. $\varphi'$ is the Wirtinger derivative of the real-valued $\varphi$. $\mathcal{R}$ is the regularizer introduced earlier, e.g., the $\ell_1$-norm.

Stacking $L$ such stages yields the full unrolled network:
\begin{equation}
f^L_\mphi(\mathbf{g}) = f_L^{\mathbf{g}, \mphi} \circ f_{L-1}^{\mathbf{g}, \mphi} \circ \dots \circ f_1^{\mathbf{g}, \mphi}(\mathbf{z}_0), \label{eq:unroll_full}
\end{equation}
with initialization $\mathbf{z}_0$ obtained via the spectral initialization from Wirtinger flow, cf. \cite[Algorithm 1]{CaLiSo15}.

To ensure boundedness of the output of the network, the 1-Lipschitz function $\sigma$
\begin{equation}
\sigma(x) =
\begin{cases}\label{sigmacases}
x & \text{if } \lVert x \rVert_2 \leq C_\text{out}, \\
C_\text{out} \dfrac{x}{\lVert x \rVert_2} & \text{otherwise},
\end{cases}
\end{equation}
is applied after mapping the sparse code through the learned dictionary $\mphi$. This way, the network output $\hat{\pmb{\psi}} = \sigma(\mphi f^L_\mphi(\mathbf{g}))$ lies within the $\ell_2$-ball of radius $C_{\text{out}}$.
For the sake of simplicity, we choose $C_\text{out} = C_\text{in}$. %

The associated hypothesis class is
\begin{equation}\label{H: spaces}
\mathcal{H}_1^L = \{ \sigma \circ (\mphi f^L_\mphi) : \mphi \in \mathcal{U}(N) \}, \quad \text{embedded in} \quad \mathcal{H}_2^L = \{ \sigma \circ (\mathbf{\Psi} f^L_\mphi) : \mphi, \mathbf{\Psi} \in \mathcal{U}(N) \}.
\end{equation}

\section{Preliminaries}

The following operator-theoretic definitions and lemmas provide the mathematical foundation to analyze the behavior of the operators underlying our network, in particular those associated with the nonlinearities resulting from the smoothed amplitude $\varphi$. Although these results are stated in real Hilbert spaces in \cite{BaCo17}, they remain valid in complex Hilbert spaces when replacing the inner product that occurs in the statements and definitions for real Hilbert spaces, by the real part of the inner product. These concepts will be used later to analyze the effect of perturbations on the unfolded network iterations.

In the following, let $\mathcal{H}$ be a (real or complex) Hilbert space.
\begin{definition}\label{defFNE} Let $D \subset \mathcal{H}$ nonempty and $T : D \to \mathcal{H}$. \begin{itemize} \item $T$ is called \emph{firmly nonexpansive}, if \begin{equation*} (\forall x \in D)(\forall y \in D)\qquad \Abs{Tx - Ty}^2 + \Abs{(\identityMapping - T)x - (\identityMapping - T)y}^2 \leq \Abs{x - y}^2. \end{equation*} \item Let $\beta \in \mathbb{R}_{++}\coloneqq(0,\infty)$. $T$ is called \emph{$\beta$-cocoercive} (or \emph{$\beta$-inverse strongly monotone}), if \begin{equation*} (\forall x \in D)(\forall y \in D)\qquad \Real\scp{x - y}{Tx - Ty} \geq \beta \Abs{Tx - Ty}^2. \end{equation*} \item Let $L \in \mathbb{R}_+\coloneqq [0,\infty)$. $T$ is called \emph{$L$-Lipschitz continuous}, if \begin{equation*} (\forall x \in D)(\forall y \in D)\qquad \Abs{Tx - Ty} \leq L \Abs{x - y}. \end{equation*} \end{itemize} \end{definition}

The properties defined above are central in analyzing proximal gradient iterations, since the gradient-like and proximal-like steps correspond to operators with Lipschitz and firmly nonexpansive behavior, respectively. By establishing these operator-theoretic properties, we can later show the effect of perturbations on the unrolled network iterations and ensure that the iterates remain bounded.

We summarize several important properties of operators and proximal mappings in $\mathcal{H}$ below.
\begin{proposition}[Properties of cocoercive and proximal operators]\label{prop1} 
  Let $D \subseteq \mathcal{H}$ be nonempty, and let $T : D \to \mathcal{H}$ be an operator. Let $\alpha, \beta \in \mathbb{R}_{++}$ and $f \in \Gamma_0(\mathcal{H})$. Then, 
  \begin{itemize} 
    \item $T$ is firmly nonexpansive if and only if it is $1$-cocoercive. \vspace{0.75em} 
    \item \label{lem:CocoerciveAndScaling} If $T$ is $\beta$-cocoercive, then $\alpha T$ is $\frac{\beta}{\alpha}$-cocoercive. \vspace{0.75em} 
    \item $T$ is $\beta$-cocoercive if and only if $\beta T$ is firmly nonexpansive. \vspace{0.75em} 
    \item If $T$ is $\beta$-cocoercive, then $T$ is $\frac{1}{\beta}$-Lipschitz continuous. \vspace{0.75em} %
    \item \label{lem:ProxFirmlyNonExpansive} $\prox_f : \mathcal{H} \to \mathcal{H}$ is firmly nonexpansive, cf. \cite[Proposition 12.28]{BaCo17}. \vspace{0.75em} 
    \item \label{lem:ProxScalingProperty} By \cite[Proposition 24.8(v)]{BaCo17}%
    , one gets the identity 
    \begin{equation} 
      \alpha\prox_f(x) = \prox_{\alpha^2 f\left(\frac{1}{\alpha} \cdot\right)}(\alpha x) \quad \text{for all } x \in \mathcal{H}.\label{eq:ProxScalingProperty} 
  \end{equation} 
  \item Let $0 \in \argmin_{x \in \mathbb{R}^n} f(x)$. 
  Then, 
    \begin{equation}\label{cor:ProxZeroAtZero} 
      \prox_f(0) = 0, 
    \end{equation} 
    as an immediate consequence of \cite[Proposition 12.29]{BaCo17} and 
    \begin{equation}\label{lem:ZeroMinOfConjugate} 
      0 \in \argmin_{x \in \mathbb{R}^n} f^*(x). 
    \end{equation} 
  \end{itemize} 
\end{proposition}

\begin{proposition}[Real variant: \cite{BaCo17}, Proposition 4.12]\label{mainprop} Let $I$ be a finite set. For every $i\in I$, let $\mathcal{K}_i$ be real or complex Hilbert space, $L_i \in \mathcal{B}(\mathcal{H}, \mathcal{K}_i)\setminus\{0\}$, $\beta_i \in \mathbb{R}_{++},$ and $T_i: \mathcal{K}_i \to \mathcal{K}_i$ be $\beta_i$-cocoercive. Moreover, let \[ T = \displaystyle\sum_{i\in I}^{} L_i^* T_i L_i \;\;\text{and}\;\; \beta = \frac{1}{\displaystyle\sum_{i\in I}^{}\frac{\Abs{L_i}^2}{\beta_i}}. \] Then, $T$ is $\beta$-cocoercive. 
\end{proposition} 

\begin{corollary}[Real variant: \cite{BaCo17}, Corollary 4.13]\label{LTLfn}
  Let $\mathcal{H}, \mathcal{K}$ be a real or complex Hilbert spaces, let $T: \mathcal{K} \to \mathcal{K}$ be firmly nonexpansive, and let $L \in \mathcal{B(\mathcal{H}, \mathcal{K})}$ be such that $\lVert L \rVert \leq 1$. Then, $L^* T L$ is firmly nonexpansive.
\end{corollary}

\begin{lemma}\label{lem:PseudoHuberDerivativeAsProx} 
  Let \begin{equation} \label{eq:PseudoHuberDerivativePreProx} \phi_1:\mathbb{R}\to\mathbb{R}\cup\{\infty\}, \quad x\mapsto \begin{cases} -\dfrac{x^2}{2}-\sqrt{1-x^2} & \text{if } \abs{x}\leq1, \\ +\infty & \text{otherwise} \end{cases} \end{equation} Then, $\phi_1 \in \Gamma_0(\mathbb{R})$, $\phi_1$ is even, $\argmin\limits_{x\in\mathbb{R}}\phi_1(x) = 0$ and $ \prox_{\phi_1}(y) = \dfrac{y}{\sqrt{1 + y^2}}. $ 
\end{lemma}

This lemma provides a key connection between the smoothed amplitude nonlinearity used in the loss function and a proximal operator, allowing us to apply the theoretical results and properties of proximal mappings in complex Hilbert spaces.

\begin{lemma}[Norm Composition {\cite[Theorem 6.18]{Be17}}]\label{lem:NormCompProx} 
  Let $\mathcal{H}$ be a real Hilbert space, and let $f : \mathcal{H} \to \mathbb{R}$ be defined by \[ f(x) = g(\|x\|), \] where $g \in \Gamma_0(\mathbb{R})$ and $\operatorname{dom}(g) \subseteq \mathbb{R}_+$. Then, the proximal operator of $f$ is given by \[ \prox_f(x) = \begin{cases} \prox_g(\Abs{x}) \dfrac{x}{\|x\|}, & \text{if } x \neq 0, \\ \{ y \in \mathcal{H} : \|y\| = \prox_g(0) \}, & \text{if } x = 0. \end{cases} \]
\end{lemma}

In the following, we use these results to explicitly represent $\varphi_\delta$ and $\varphi_\delta'$ as proximal operators in the complex setting (\cref{proxstar}), which makes it possible to formulate the unfolded iterations in the form of proximal operators.

\begin{remark}\label{genproxg} 
  Let $g\coloneqq\phi_1+I_{\mathbb{R}_+}$ with $\phi_1$ from \eqref{eq:PseudoHuberDerivativePreProx} and $I_{\mathbb{R}_+}(t)=
  \begin{cases} 
    0&t\in\mathbb{R}_+\\ \infty&t\not\in\mathbb{R}_+ 
  \end{cases}$.
   Then, $g\in\Gamma_0(\mathbb{R})$ and for $t\in\mathbb{R}_+$ 
   \begin{align*} 
    \prox_g(t)&=\prox_{\phi_1+I_{\mathbb{R}_+}}(t)=\argmin_{s\in\mathbb{R}}\brac{\phi_1(s)+I_{\mathbb{R}_+}(s)+\tfrac12(s-t)^2}\\ &=\argmin_{s\in\mathbb{R}_+}\brac{\phi_1(s)+\tfrac12(s-t)^2}=P_{\mathbb{R}_+}\brac{\argmin_{s\in\mathbb{R}}\brac{\phi_1(s)+\tfrac12(s-t)^2}}\\ &=P_{\mathbb{R}_+}\brac{\prox_{\phi_1}(t)}\overset{\textup{\cref{lem:PseudoHuberDerivativeAsProx}}}{=}P_{\mathbb{R}_+}\brac{\frac{t}{\sqrt{1+t^2}}}\overset{t\in\mathbb{R}_+}=\frac{t}{\sqrt{1+t^2}}. 
  \end{align*} 
  Here, $P_{\mathbb{R}_+}$ denotes the Euclidean projection from $\mathbb{R}$ to $\mathbb{R_+}$. From this, we get 
  \[
  \{ y \in \mathbb{R}^n : \Abs{y} = \prox_g(0) \}=\{ y \in \mathbb{R}^n : \Abs{y} = 0 \}=\{0\}.
  \]
  For $x\in\mathbb{R}^n\setminus\{0\}$, we get 
  \[
  \prox_g(\Abs{x})\frac{x}{\Abs{x}}=\frac{\Abs{x}}{\sqrt{1+\Abs{x}^2}}\frac{x}{\Abs{x}}=\frac{x}{\sqrt{1+\Abs{x}^2}}.
  \] 
  Combined with the above, we get from \cref{lem:NormCompProx} for $f : \mathbb{R}^n \to \mathbb{R},\ x\mapsto g(\Abs{x})$ that 
  \[ 
  \prox_f (x) = \begin{cases} \prox_g(\Abs{x}) \frac{x}{\Abs{x}}, & x \neq 0, \\ 0, & x = 0. \end{cases}=\frac{x}{\sqrt{1+\Abs{x}^2}}. 
  \] 
  Moreover, $f\in\Gamma_0(\mathbb{R}^n)$, since $\phi_1\in\Gamma_0(\mathbb{R})$ and thus $g\in\Gamma_0(\mathbb{R})$.
 \end{remark}

\begin{remark}\label{proxstar}
  Let $\varphi_\delta:\mathbb{C}\to\mathbb{R},\ z\mapsto\sqrt{|z|^2+\delta^2}-\delta$ for $\delta>0$ and identify $\mathbb{C}$ with $\mathbb{R}^2$. Then, we have for the Wirtinger derivative of $\varphi_\delta$ that 
  \begin{align*} 
    \varphi_\delta'(z)={}&\frac{z}{\sqrt{|z|^2+\delta^2}}=\frac{\frac{z}{\delta}}{\sqrt{\abs{\frac{z}{\delta}}^2+1}}\overset{\textup{\cref{genproxg}}}=\prox_{f}\brac{\frac{z}{\delta}} \overset{\eqref{eq:ProxScalingProperty}}{=}\frac{1}{\delta}\prox_{\delta^2f(\frac{\cdot}{\delta})}\brac{z}. 
  \end{align*}
  Moreover, we have \[\varphi_\delta(z)\varphi_\delta'(z)=\brac{\sqrt{|z|^2+\delta^2}-\delta}\frac{z}{\sqrt{|z|^2+\delta^2}}=z-\frac{\delta z}{\sqrt{|z|^2+\delta^2}}=z-\delta\varphi_\delta'(z)=z-\prox_{\delta^2f(\frac{\cdot}{\delta})}\brac{z}\] Let \begin{equation}\label{eq:fDeltaDef} f_\delta:\mathbb{C}\to\mathbb{R},\ z\mapsto\delta^2f\brac{\frac{z}{\delta}}. \end{equation} Since $f\in\Gamma_0(\mathbb{C})$, we also have $f_\delta\in\Gamma_0(\mathbb{C})$. Thus, by \cite[Corollary 13.38]{BaCo17}, $f_\delta^*\in\Gamma_0(\mathbb{C})$, where $f_\delta^*$ denotes the convex conjugate (or Fenchel conjugate) of $f_\delta$. By Moreau's decomposition (\cite[Theorem 2.1]{Co18}), we get \begin{equation}\label{eq:MoreauDecomposition} z-\prox_{f_\delta}\brac{z}=\prox_{f_\delta^*}\brac{z}, \end{equation} In total, we have shown \begin{equation} \label{eq:NonlinearitiesAsProx} \varphi_\delta'(z)=\frac{1}{\delta}\prox_{f_\delta}\brac{z}\text{ and }\varphi_\delta(z)\varphi_\delta'(z)=\prox_{f_\delta^*}\brac{z}. \end{equation} Again using \eqref{eq:MoreauDecomposition}, we get for $g\in\mathbb{R}$, \[\varphi_\delta(z)\varphi_\delta'(z)-g\varphi_\delta'(z)=z-\prox_{f_\delta}\brac{z}-\frac{g}{\delta}\prox_{f_\delta}\brac{z}=z-\brac{1+\frac{g}{\delta}}\prox_{f_\delta}\brac{z}.\] Moreover, since $f_\delta\in\Gamma_0(\mathbb{C})$ and $0 \in \displaystyle\argmin_{z \in \mathbb{C}} f_\delta(z)$, by \eqref{cor:ProxZeroAtZero}, we get $\prox_{f_\delta}\brac{0} = 0$. Moreover, using \eqref{lem:ZeroMinOfConjugate}, we can analogously conclude $\prox_{f_\delta^*}\brac{0} = 0$. 
\end{remark}
\begin{remark}
  From \cref{proxstar}, one can obtain by taking the modulus of $\varphi_\delta'$
\begin{equation}\label{supnablaphi}
|\varphi_\delta'(z)| = \frac{|z|}{\sqrt{|z|^2 + \delta^2}} = \frac{\sqrt{|z|^2}}{\sqrt{|z|^2 + \delta^2}} \le \frac{\sqrt{|z|^2 + \delta^2}}{\sqrt{|z|^2 + \delta^2}} = 1.\qedhere
\end{equation}
\end{remark}

\section{Exponential Growth}
In the following, we establish a bound on the network output as well as a pertubation bound for the network parameters. The latter shows that the network output is Lipschitz with respect to the network parameters.
\begin{assumption}\label{assum1}
 Let $\mphi \in \mathcal{U}(\numpixel)$, $\mathbf{A}\in\mathbb{C}^{KN\times N}$ and $\tau_\ell \in \mathbb{R}_{>0}$ with
  \begin{equation}\label{bddest}
  \frac{\tau_\ell}{KN} \Abs{ \mathbf{A} }^2_{2 \to 2} \leq 1.
  \end{equation} 
\end{assumption}
\begin{definition}\label{assum2}%
For $z\in\mathbb{C}^N$, define the operators
  \begin{equation}
    \label{eq:DefTPhiAnsSGPhi}
  T_{\mphi}(z) \coloneqq \frac{1}{KN} (\mathbf{A}\mphi)^* T\left( \mathbf{A}\mphi z \right)\quad \text{and}\quad
  S^{\mathbf{g}}_{\mphi}(z) \coloneqq \frac{1}{KN} (\mathbf{A}\mphi)^* S^{\mathbf{g}}\left( \mathbf{A}\mphi z \right),
  \end{equation}
  where
  \[
    T(z) \coloneqq \prox_{f_\delta^*}(z)\quad \text{and}\quad S^{\mathbf{g}}(z) \coloneqq \frac{\mathbf{g}}{\delta}\odot \prox_{f_\delta}(z).
  \]
\end{definition}
\begin{remark}\label{tandsvanish}
  By \cref{assum2}, we have for $z \in \mathbb{C}^{\numpixel}$ that
  \[
  T_{\mphi}(z) = \frac{1}{KN} (\mathbf{A}\mphi)^* \prox_{f_\delta^*}\left( \mathbf{A}\mphi z \right)\quad \text{and}\quad
  S^{\mathbf{g}}_{\mphi}(z) = \frac{1}{KN} (\mathbf{A}\mphi)^* \frac{\mathbf{g}}{\delta}\odot \prox_{f_\delta}\left( \mathbf{A}\mphi z \right),
  \]
Moreover, since $\prox_{f_\delta}$ and $\prox_{f_\delta^*}$ vanish at zero (cf. \cref{proxstar}), we get
\[
S^{\mathbf{g}}_{\mphi}(0)=\prox_{f_\delta}\brac{0} = 0 = \prox_{f_\delta^*}\brac{0} = T_{\mphi}(0).
\] 
\end{remark}
\begin{remark}\label{variousineq}
The following inequalities will be used later.
\begin{itemize}
    \item Let $\mathbf{A} \in \mathbb{C}^{m \times n}$ and $\mathbf{z} \in \mathbb{C}^{n}$. Then,
     \begin{equation}
        \Abs{\mathbf{Az}}_2 \leq \Abs{\mathbf{A}}_{2 \to 2} \Abs{\mathbf{z}}_2,
        \label{eq:first}
    \end{equation}
    which follows directly from the definition of the operator norm.

    \vspace{0.75em}

    \item Let $\mathbf{x}, \mathbf{y} \in \mathbb{C}^{m \times n}$. Then,
    \begin{equation}
        \Abs{\mathbf{x} \odot \mathbf{y}}_F^2 = \sum_{i=1}^{m} \sum_{j=1}^{n} |x_{ij} y_{ij}|^2 
        \leq \max_{i,j} |x_{ij}|^2 \sum_{i=1}^{m} \sum_{j=1}^{n} |y_{ij}|^2 
        = \Abs{\mathbf{x}}_\infty^2 \Abs{\mathbf{y}}_F^2,
        \label{eq:third}
    \end{equation}
    where we used the entrywise inequality $|x_{ij} y_{ij}|^2 \leq \left( \max_{i,j} |x_{ij}|^2 \right) |y_{ij}|^2$.
\end{itemize}

\vspace{0.75em}

In particular, for vectors $\mathbf{x}, \mathbf{y} \in \mathbb{C}^{n}$, we have
\begin{equation}
    \Abs{\mathbf{x} \odot \mathbf{y}}_2^2 \leq \Abs{\mathbf{x}}_\infty^2 \Abs{\mathbf{y}}_2^2.
    \label{eq:second}
\end{equation}
\end{remark}

\begin{lemma}\label{TandS}
  If \cref{assum1} holds, $\tau_\ell T_\mphi$ is firmly nonexpansive and $\tau_\ell S^{\mathbf{g}}_\mphi$ is $\frac{\Abs{\mathbf{g}}_\infty}{\delta}$-Lipschitz continuous.
  \end{lemma}
  \begin{proof}
  Recalling \eqref{eq:DefTPhiAnsSGPhi}, we have
  \[
  KN T_\mphi(z) = (\mathbf{A}\mphi)^* \prox_{f_\delta^*}( \mathbf{A}\mphi z ) =  L^* \prox_{f_\delta^*}(Lz),
  \]
  where $L = \mathbf{A}\mphi$. By applying \cref{mainprop} and noting that $\prox_{f_\delta^*}$ is 1-cocoercive and $\Abs{\mathbf{A}\mphi}_{2 \to 2}=\Abs{\mathbf{A}}_{2 \to 2}$, we obtain that
  $
  KN T_\mphi \text{ is } \frac{1}{\Abs{\mathbf{A}}_{2 \to 2}^2}\text{-cocoercive}.
  $
  Thus, by using \cref{lem:CocoerciveAndScaling} with $\alpha = \frac{1}{K\numpixel}$%
  , we obtain that
  $
  T_\mphi \text{ is } \frac{KN}{\Abs{\mathbf{A}}_{2 \to 2}^2}\text{-cocoercive}$. Hence $\tau_\ell T_\mphi \text{ is } \frac{KN}{\tau_{\ell}\Abs{\mathbf{A}}_{2 \to 2}^2}\text{-cocoercive}.
  $
  If \cref{assum1} holds, we get $\frac{KN}{\tau_{\ell}\Abs{\mathbf{A}}_{2 \to 2}^2}\geq1$. Thus, $\tau_\ell T_\mphi$  is  $1$-cocoercive, and thus firmly nonexpansive. 
  
  For the Lipschitz continuity of $\tau_\ell S^{\mathbf{g}}_\mphi$, let $z_1,z_2\in\mathbb{C}^N$. Then,
  \[
  \Abs{ S^{\mathbf{g}}_\mphi(z_1) - S^{\mathbf{g}}_\mphi(z_2) }_2 = \Abs{ \frac{1}{KN}(\mathbf{A}\mphi)^* \left( \frac{\mathbf{g}}{\delta} \odot \left( \prox_{f_\delta}(\mathbf{A}\mphi z_1) - \prox_{f_\delta}(\mathbf{A}\mphi z_2) \right) \right) }_2.
  \]
  Since $\prox_{f_\delta}$ is 1-Lipschitz and using \eqref{eq:first} and \eqref{eq:second}, we get
  \begin{align*}
  \Abs{ S^{\mathbf{g}}_\mphi(z_1) - S^{\mathbf{g}}_\mphi(z_2) }_2 &\leq \frac{1}{KN} \Abs{\mathbf{A}\mphi^*}_{2 \to 2}\Abs{ \frac{\mathbf{g}}{\delta} \odot \left( \prox_{f_\delta}(\mathbf{A}\mphi z_1) - \prox_{f_\delta}(\mathbf{A}\mphi z_2) \right) }_2 \\
  &\leq \frac{1}{KN} \Abs{\mathbf{A}}_{2 \to 2} \Abs{ \frac{\mathbf{g}}{\delta} }_\infty \Abs{ \prox_{f_\delta}(\mathbf{A}\mphi z_1) - \prox_{f_\delta}(\mathbf{A}\mphi z_2) }_2 \\
  &\leq \frac{1}{KN} \Abs{\mathbf{A}}_{2 \to 2} \Abs{ \frac{\mathbf{g}}{\delta} }_\infty \Abs{ \mathbf{A}\mphi z_1 - \mathbf{A}\mphi z_2 }_2 \\
  &\leq \frac{1}{KN} \Abs{\mathbf{A}}_{2 \to 2} \Abs{ \frac{\mathbf{g}}{\delta} }_\infty \Abs{\mathbf{A}\mphi}_{2 \to 2} \Abs{ z_1 -  z_2 }_2 \\
  &= \frac{1}{KN} \Abs{\mathbf{A}}^2_{2 \to 2} \Abs{ \frac{\mathbf{g}}{\delta} }_\infty \Abs{ z_1 -  z_2 }_2.
  \end{align*}
  Using \cref{assum1}, we get that
  $
  \tau_\ell S^{\mathbf{g}}_\mphi \text{ is } \frac{\Abs{\mathbf{g}}_\infty}{\delta}\text{-Lipschitz}.
  $
  \end{proof}
Before stating a corollary from the above, we extend the operators $T_\mphi$ and $S_\mphi^\mathbf{G}$ 
from $\mathbb{C}^N$ to $\mathbb{C}^{N \times m}$ by applying them column-wise. 
For $\mathbf{Z} = (\mathbf{z}_1 \mid \cdots \mid \mathbf{z}_m) \in \mathbb{C}^{N \times m}$, we set
\[
\tau_\ell T_\mphi(\mathbf{Z}) \coloneqq \big( \tau_\ell T_\mphi(\mathbf{z}_1) \mid \cdots \mid \tau_\ell T_\mphi(\mathbf{z}_m) \big),
\]
and similarly, for $\mathbf{G} = (\mathbf{g}_1 \mid \cdots \mid \mathbf{g}_m) \in \mathbb{C}^{KN \times m}$,
\[
\tau_\ell S_\mphi^\mathbf{G}(\mathbf{Z}) \coloneqq \big( \tau_\ell S_\mphi^{\mathbf{g}_1}(\mathbf{z}_1) \mid \cdots \mid \tau_\ell S_\mphi^{\mathbf{g}_m}(\mathbf{z}_m) \big).
\]
These extensions are consistent with the Frobenius norm since
\[
\Abs{\mathbf{Z}}_F^2 = \sum_{j=1}^m \Abs{\mathbf{z}_j}_2^2.
\]
Moreover, since $T_\mphi(0) = 0$ and $S_\mphi^{\mathbf{g}}(0) = 0$ by \cref{tandsvanish}, 
it follows that their column-wise extensions also vanish at $0 \in \mathbb{C}^{N \times m}$.
\begin{corollary}\label{SandTfnlc}
  Let $\mathbf{G} = \begin{pmatrix} \mathbf{g}_1 \mid \cdots \mid \mathbf{g}_m \end{pmatrix} \in \mathbb{C}^{KN \times m}$ and $\mathbf{Z} = \begin{pmatrix} \mathbf{z}_1 \mid \cdots \mid \mathbf{z}_m \end{pmatrix} \in \mathbb{C}^{N \times m}$. Under \cref{assum1}, the following hold:
  \begin{itemize}
    \item The operator $\tau_\ell T_\mphi : \mathbb{C}^{N \times m} \to \mathbb{C}^{N \times m}$ is firmly nonexpansive with respect to the Frobenius norm. Moreover, we have
    \[
    \Abs{(\identityMapping -\tau_\ell T_\mphi)(\mathbf{Z})}_F \leq \Abs{\mathbf{Z}}_F.
    \]

    \item The operator $\tau_\ell S^\mathbf{G}_\mphi : \mathbb{C}^{N \times m} \to \mathbb{C}^{N \times m}$ is $\Abs{\mathbf{G}}_\infty / \delta$-Lipschitz continuous with respect to the Frobenius norm. Moreover, we have
    \[
    \Abs{\tau_\ell S^\mathbf{G}_\mphi(\mathbf{Z})}_F \leq \frac{\Abs{\mathbf{G}}_\infty}{\delta} \Abs{\mathbf{Z}}_F.
    \]
  \end{itemize}
\end{corollary}
\begin{proof}
  From \cref{TandS}, the operator $\tau_\ell T_\mphi : \mathbb{C}^N \to \mathbb{C}^N$ is firmly nonexpansive. 
  By construction, the extension to $\mathbb{C}^{N \times m}$ is column-wise and thus respects the Frobenius norm. 
  Explicitly, for $\mathbf{Z} = (\mathbf{z}_1 \mid \cdots \mid \mathbf{z}_m)$, we obtain
  \[
  \Abs{(\identityMapping - \tau_\ell T_\mphi)(\mathbf{Z})}_F^2 
  = \sum_{j=1}^m \Abs{(\identityMapping - \tau_\ell T_\mphi)(\mathbf{z}_j)}_2^2 
  \leq \sum_{j=1}^m \Abs{\mathbf{z}_j}_2^2 
  = \Abs{\mathbf{Z}}_F^2,
  \]
  which shows the Frobenius nonexpansiveness directly.

  For the second part, \cref{TandS} states that for any $z_1, z_2 \in \mathbb{C}^N$ and corresponding $\mathbf{g} \in \mathbb{C}^{KN}$,
  \[
  \Abs{\tau_\ell S^{\mathbf{g}}_\mphi(z_1) - \tau_\ell S^{\mathbf{g}}_\mphi(z_2)}_2 \leq \frac{\Abs{\mathbf{g}}_\infty}{\delta} \Abs{z_1 - z_2}_2.
  \]
  Applying this column-wise to $\mathbf{Z}_1 = \begin{pmatrix} \mathbf{z}_{1,1} \mid \cdots \mid \mathbf{z}_{1,m} \end{pmatrix}$ and $\mathbf{Z}_2 = \begin{pmatrix} \mathbf{z}_{2,1} \mid \cdots \mid \mathbf{z}_{2,m} \end{pmatrix}$, we get
  \begin{align*}
    \Abs{\tau_\ell S^\mathbf{G}_\mphi(\mathbf{Z}_1) - \tau_\ell S^\mathbf{G}_\mphi(\mathbf{Z}_2)}_F^2
    &= \sum_{j=1}^m \Abs{\tau_\ell S^{\mathbf{g}_j}_\mphi(\mathbf{z}_{1,j}) - \tau_\ell S^{\mathbf{g}_j}_\mphi(\mathbf{z}_{2,j})}_2^2 
    \leq \sum_{j=1}^m \left( \frac{\Abs{\mathbf{g}_j}_\infty}{\delta} \right)^2 \Abs{\mathbf{z}_{1,j} - \mathbf{z}_{2,j}}_2^2 \\
    &\leq \left( \frac{\Abs{\mathbf{G}}_\infty}{\delta} \right)^2 \sum_{j=1}^m \Abs{\mathbf{z}_{1,j} - \mathbf{z}_{2,j}}_2^2 
    = \left( \frac{\Abs{\mathbf{G}}_\infty}{\delta} \right)^2 \Abs{\mathbf{Z}_1 - \mathbf{Z}_2}_F^2.
  \end{align*}
  By \cref{tandsvanish}, we have $S^{\mathbf{g}}_\mphi(0) = 0$ for each $\mathbf{g}$, which implies
  \[
  S^\mathbf{G}_\mphi(0) = \begin{pmatrix} S^{\mathbf{g}_1}_\mphi(0) \mid \cdots \mid S^{\mathbf{g}_m}_\mphi(0) \end{pmatrix} = 0.
  \]
  Taking $\mathbf{Z}_1 = \mathbf{Z}$ and $\mathbf{Z}_2 = 0$ in the Lipschitz inequality, we obtain:
  \[
  \Abs{\tau_\ell S^\mathbf{G}_\mphi(\mathbf{Z})}_F = \Abs{\tau_\ell S^\mathbf{G}_\mphi(\mathbf{Z}) - \tau_\ell S^\mathbf{G}_\mphi(0)}_F \leq \frac{\Abs{\mathbf{G}}_\infty}{\delta} \Abs{\mathbf{Z}}_F. \qedhere
  \]
\end{proof}
\begin{corollary}\label{3rdtermbound}
  Under \cref{assum1}, for any $z_1,z_2 \in \mathbb{C}^\numpixel$, we have
  \begin{equation}\label{expr0}
    \Abs{(z_1 - z_2) - \tau_\ell \left[\left( T_\mphi - S^{\mathbf{g}}_\mphi \right)(z_1) - \left( T_\mphi - S^{\mathbf{g}}_\mphi \right)(z_2)\right]}_2
    \leq \left(1 + \frac{\Abs{\mathbf{g}}_\infty}{\delta} \right) \Abs{z_1 - z_2}_2.
  \end{equation}
  \end{corollary}
  
  \begin{proof}
    This is an immediate consequence of \cref{TandS} and that the firmly nonexpansiveness of an operator $T$ implies the nonexpansiveness of $\identityMapping-T$.
  \end{proof} 
\begin{remark}\label{def: matrixiteration}
Let $\pmb{\psi}^0 \in \mathbb{C}^N$ with $\lVert\pmb{\psi}^0\rVert_2 = 1$. For $\mphi \in \mathcal{U}(\numpixel)$, $\ell \in \mathbb{N}_0$ and $\mathbf{g} \in \mathbb{C}^{KN}$, we define
\[
f_\mphi^{\ell+1}(\mathbf{g}) :=
\prox_{\tau_\ell \mathcal{R}} \left( 
f_\mphi^{\ell}(\mathbf{g}) 
- \frac{\tau_\ell}{KN} \left(\overline{\mathbf{A}\mphi}\right)^{\!T} 
\left( \varphi_\delta\!\left(\mathbf{A}\mphi f_\mphi^{\ell}(\mathbf{g})\right) 
- \mathbf{g} \right) 
\odot \varphi_\delta'\!\left(\mathbf{A}\mphi f_\mphi^{\ell}(\mathbf{g})\right) 
\right).
\]
with initialization
\[
f_\mphi^0(\mathbf{g}) := \pmb{\psi}^0 \in \mathbb{C}^N, 
\quad \text{with } \lVert\pmb{\psi}^0\rVert_2 = 1.
\]
For a matrix $\mathbf{G} = \begin{pmatrix} \mathbf{g}_1 \mid \cdots \mid \mathbf{g}_m \end{pmatrix} \in \mathbb{C}^{KN \times m}$,  
we extend this definition column-wise, i.e., for $\ell\in\mathbb{N}_0$, we define
\[
f_\mphi^{\ell}(\mathbf{G}) :=
\begin{pmatrix}
f_\mphi^{\ell}(\mathbf{g}_1) \mid \cdots \mid f_\mphi^{\ell+1}(\mathbf{g}_m)
\end{pmatrix}.
\]
This implies,
\[
\mathbf{f} := f_\mphi^0(\mathbf{G}) = 
\begin{pmatrix} \pmb{\psi}^0 \mid \cdots \mid \pmb{\psi}^0 \end{pmatrix} 
\in \mathbb{C}^{N \times m},
\]
i.e., the same initialization $\pmb{\psi}^0$ is used for each column. Moreover,
\[
\lVert\mathbf{f}\rVert_F^2 = 
\left\lVert
\begin{pmatrix} 
\pmb{\psi}^0 \mid \cdots \mid \pmb{\psi}^0 
\end{pmatrix}
\right\rVert_F^2
= \sum_{j=1}^m \lVert \pmb{\psi}^0\rVert_2^2 = m.
\]
\end{remark}

\begin{remark}\label{ineqoutput}
Using \eqref{eq:NonlinearitiesAsProx} element-wise, we get the identity
\begin{equation*}
\left(\varphi_\delta(z) - \mathbf{g}\right) \odot \varphi_\delta'(z) = \varphi_\delta(z)\odot \varphi_\delta'(z) - \mathbf{g}\odot \varphi_\delta'(z)= \prox_{f_\delta^*}\brac{z} - \frac{\mathbf{g}}{\delta} \odot \prox_{f_\delta}(z).
\end{equation*}
Combined with $T_{\mphi}, S^{\mathbf{g}}_{\mphi}$ from \cref{assum2}, the iteration \( f_\mphi^{\ell+1}(\mathbf{g}) \) can be expressed as:
\begin{align}
f_\mphi^{\ell+1}(\mathbf{g}) &= \prox_{\tau_\ell \mathcal{R}} \left( f_\mphi^{\ell}(\mathbf{g}) - \frac{\tau_\ell}{KN} (\mathbf{A}\mphi)^* \left( \varphi_\delta(\mathbf{A}\mphi f_\mphi^{\ell}(\mathbf{g})) - \mathbf{g} \right) \odot \varphi_\delta'(\mathbf{A}\mphi f_\mphi^{\ell}(\mathbf{g})) \right) \notag \\
&=\prox_{\tau_\ell \mathcal{R}} \left( f_\mphi^{\ell}(\mathbf{g}) - \frac{\tau_\ell}{KN} (\mathbf{A}\mphi)^* \left( \prox_{f_\delta^*}(\mathbf{A}\mphi f_\mphi^{\ell}(\mathbf{g})) - \frac{\mathbf{g}}{\delta} \odot \prox_{f_\delta}(\mathbf{A}\mphi f_\mphi^{\ell}(\mathbf{g})) \right)\right) \notag \\
&= \prox_{\tau_\ell \mathcal{R}} \left( f_\mphi^{\ell}(\mathbf{g}) - \frac{\tau_\ell}{KN} (\mathbf{A}\mphi)^* \left( T(\mathbf{A}\mphi f_\mphi^{\ell}(\mathbf{g})) - S^{\mathbf{g}}(\mathbf{A}\mphi f_\mphi^{\ell}(\mathbf{g})) \right) \right)\label{iterationFormLong}\\
&= \prox_{\tau_\ell \mathcal{R}} \left( f_\mphi^{\ell}(\mathbf{g}) - \tau_\ell \left( T_{\mphi}(f_\mphi^{\ell}(\mathbf{g})) - S^{\mathbf{g}}_{\mphi}(f_\mphi^{\ell}(\mathbf{g})) \right) \right) \label{desiredeq}.
\end{align}
\end{remark}
\begin{lemma}\label{defG}
  Let $\mathbf{G} \in \mathbb{R}^{KN \times m}$ and $\mphi \in \mathcal{U}(\numpixel)$. Furthermore, let $\mathcal{R} \in \Gamma_0 (\mathbb{C})$ with $0 \in \displaystyle\argmin_{x \in \mathbb{R}^n} \mathcal{R}(x)$. Then, if \cref{assum1} holds, we have
  \begin{equation}\label{out1}
    \Abs{f_{\mphi}^{\ell}(\mathbf{G})}_F \leq \sqrt{m} + \frac{1}{KN} \Abs{\mathbf{A}}_{2 \to 2}\cdot \Abs{\G}_\infty \displaystyle\sum_{k=0}^{\ell-1}\tau_k \quad\text{  for }\ell \in \mathbb{N}_{0}.
  \end{equation}
\end{lemma}
\begin{proof}
Let $k\in\mathbb{N}$. Using \eqref{desiredeq}, the nonexpansiveness of proximal mappings, $\prox_{\tau_{k} \mathcal{R}}(0) = 0$ (cf. \eqref{cor:ProxZeroAtZero}, \cref{prop1}), and \cref{SandTfnlc}, we obtain that
\begin{align*}
  \Abs{f_\mphi^{k+1}(\mathbf{G})}_F&\leq \Abs{f_\mphi^{k}(\mathbf{G}) - \tau_k \left( T_{\mphi}(f_\mphi^{k}(\mathbf{G})) - S^{\mathbf{G}}_{\mphi}(f_\mphi^{k}(\mathbf{G})) \right) }_F\\
  &\leq\Abs{  (\identityMapping - \tau_k T_\mphi)(f_\mphi^{k}(\mathbf{G}))}_F + \Abs{\tau_k S^{\mathbf{G}}_{\mphi}(f_\mphi^{k}(\mathbf{G}))}_F\\
  &\leq \left\| f_{\mphi}^{k}(\mathbf{G}) \right\|_F + \Abs{\tau_k S^{\mathbf{G}}_{\mphi}(f_\mphi^{k}(\mathbf{G}))}_F\\
  &= \left\| f_{\mphi}^{k}(\mathbf{G}) \right\|_F + \Abs{\frac{\tau_k}{KN} (\mathbf{A} \mathbf{\Phi})^{*} \big(\mathbf{G} \odot \varphi'_\delta(\mathbf{A} \mathbf{\Phi} f_\mathbf{\Phi}^{k}(\mathbf{G}))\big)}_F
\end{align*}
Since $\varphi'_{\delta}(\mathbf{A} \mathbf{\Phi} f_\mathbf{\Phi}^{k}(\mathbf{G})) \leq L_{\varphi_{\delta}} =\displaystyle\sup_{x \in \mathbb{R}^n} \Abs{\varphi'_\delta(x)}_2 \overset{\eqref{supnablaphi}}{=} 1$, we obtain
\begin{align}\label{bind}\nonumber
  \Abs{f_\mphi^{k+1}(\mathbf{G})}_F &\leq \left\| f_{\mphi}^{k}(\mathbf{G}) \right\|_F + \frac{\tau_k}{KN} \Abs{\mathbf{A}}_{2 \to 2}\cdot L_{\varphi_{\delta}}\Abs{\G}_\infty\\ 
   &= \left\| f_{\mphi}^{k}(\mathbf{G}) \right\|_F + \frac{\tau_k}{KN} \Abs{\mathbf{A}}_{2 \to 2}\cdot\Abs{\G}_\infty 
\end{align}
Note that above we use the boundedness of $\varphi'_{\delta}$ instead of its Lipschitz continuity to avoid getting $c\left\| f_{\mphi}^{k}(\mathbf{G}) \right\|_F$ with $c>1$ as bound for this step, which would lead to a bound for $\Abs{f_{\mphi}^{\ell}(\mathbf{G})}_F$ that grows exponentially with $\ell$.

By starting with $\Abs{f_{\mphi}^{\ell}(\mathbf{G})}_F$ and applying the above inequality $\ell$-times, we get
\begin{equation*}
  \Abs{f_{\mphi}^{\ell}(\mathbf{G})}_F \leq \Abs{f_{\mphi}^{0}(\mathbf{G})}_F + \frac{1}{KN} \Abs{\mathbf{A}}_{2 \to 2}\cdot \Abs{\G}_\infty \displaystyle\sum_{k=0}^{\ell-1}\tau_k.
\end{equation*}
Combined with $\left\| f_{\mphi}^0(\mathbf{G}) \right\|_F = \|\mathbf{f}\|_F = \sqrt{m}$ (cf. \cref{def: matrixiteration}), this shows the statement.
\end{proof}
\begin{lemma}\label{perturb}
  Let $\mathbf{G} = \begin{pmatrix} \mathbf{g}_1 \mid \cdots \mid \mathbf{g}_m \end{pmatrix} \in \mathbb{R}^{KN \times m}$ and $\DictionaryMatOne, \DictionaryMatTwo \in \mathcal{U}(\numpixel)$. If \cref{assum1} holds, then
\begin{align*}
  &\Abs{f_{\DictionaryMatOne}^{\ell +1}(\mathbf{G}) - f_{\DictionaryMatTwo}^{\ell +1}(\mathbf{G})}_F \\
  &\leq \left(1 + \frac{\|\mathbf{G}\|_\infty}{\delta}\right) \left(\frac{2\tau_\ell}{KN} \Abs{\mathbf{A}}_{2 \to 2} \Abs{\mathbf{A}\DictionaryMatTwo - \mathbf{A}\DictionaryMatOne}_{2 \to 2} \Abs{ f_{\DictionaryMatOne}^{\ell}(\mathbf{G}) }_F + \Abs{ f_{\DictionaryMatOne}^{\ell}(\mathbf{G}) - f_{\DictionaryMatTwo}^{\ell}(\mathbf{G}) }_F\right).
\end{align*}
\end{lemma}
\begin{proof}
  Starting from the iteration identity \eqref{iterationFormLong} applied columnwise, we bound the difference:
  \begin{equation*}
    \begin{split}
    \Abs{f_{\DictionaryMatOne}^{\ell+1}(\mathbf{G}) - f_{\DictionaryMatTwo}^{\ell+1}(\mathbf{G})}_F ={}& \Abs{\prox_{\tau_\ell \mathcal{R}} \left( f_{\DictionaryMatOne}^{\ell}(\mathbf{G}) - \frac{\tau_\ell}{KN} (\mathbf{A}\DictionaryMatOne)^*\left( T - S^{\mathbf{G}} \right) (\mathbf{A}\DictionaryMatOne f_{\DictionaryMatOne}^{\ell}(\mathbf{G}))\right) \right.\\
    &\left.- \prox_{\tau_\ell \mathcal{R}} \left( f_{\DictionaryMatTwo}^{\ell}(\mathbf{G}) - \frac{\tau_\ell}{KN} (\mathbf{A}\DictionaryMatTwo)^*\left( T - S^{\mathbf{G}} \right) (\mathbf{A}\DictionaryMatTwo f_{\DictionaryMatTwo}^{\ell}(\mathbf{G}))\right)}_F.
  \end{split} 
\end{equation*} 
Adding cross terms %
and using the triangle inequality, we obtain
\begin{equation*}%
  \begin{split}
    \Abs{f_{\DictionaryMatOne}^{\ell+1}(\mathbf{G}) - f_{\DictionaryMatTwo}^{\ell+1}(\mathbf{G})}_F 
    \leq{}& \Abs{
      \prox_{\tau_\ell \mathcal{R}}\left( f_{\DictionaryMatOne}^\ell(\mathbf{G}) 
      - \frac{\tau_\ell}{KN} (\mathbf{A}\DictionaryMatOne)^*(T - S^{\mathbf{G}})(\mathbf{A}\DictionaryMatOne f_{\DictionaryMatOne}^\ell(\mathbf{G})) \right)
      \right. \\
    &\left.
      - \prox_{\tau_\ell \mathcal{R}}\left( f_{\DictionaryMatOne}^\ell(\mathbf{G}) 
      - \frac{\tau_\ell}{KN} (\mathbf{A}\DictionaryMatOne)^*(T - S^{\mathbf{G}})(\mathbf{A}\DictionaryMatTwo f_{\DictionaryMatOne}^\ell(\mathbf{G})) \right)
    }_F \\
    &+ \Abs{
      \prox_{\tau_\ell \mathcal{R}}\left( f_{\DictionaryMatOne}^\ell(\mathbf{G}) 
      - \frac{\tau_\ell}{KN} (\mathbf{A}\DictionaryMatOne)^*(T - S^{\mathbf{G}})(\mathbf{A}\DictionaryMatTwo f_{\DictionaryMatOne}^\ell(\mathbf{G})) \right)
      \right. \\
    &\left.
      - \prox_{\tau_\ell \mathcal{R}}\left( f_{\DictionaryMatOne}^\ell(\mathbf{G}) 
      - \frac{\tau_\ell}{KN} (\mathbf{A}\DictionaryMatTwo)^*(T - S^{\mathbf{G}})(\mathbf{A}\DictionaryMatTwo f_{\DictionaryMatOne}^\ell(\mathbf{G})) \right)
    }_F \\
    &+ \Abs{
      \prox_{\tau_\ell \mathcal{R}}\left( f_{\DictionaryMatOne}^\ell(\mathbf{G}) 
      - \frac{\tau_\ell}{KN} (\mathbf{A}\DictionaryMatTwo)^*(T - S^{\mathbf{G}})(\mathbf{A}\DictionaryMatTwo f_{\DictionaryMatOne}^\ell(\mathbf{G})) \right)
      \right. \\
    &\left.
      - \prox_{\tau_\ell \mathcal{R}}\left( f_{\DictionaryMatTwo}^\ell(\mathbf{G}) 
      - \frac{\tau_\ell}{KN} (\mathbf{A}\DictionaryMatTwo)^*(T - S^{\mathbf{G}})(\mathbf{A}\DictionaryMatTwo f_{\DictionaryMatTwo}^\ell(\mathbf{G})) \right)
    }_F.
  \end{split}
\end{equation*}
Using the 1-Lipschitz property of the proximal operator, we get
\begin{align*}
  &\Abs{f_{\DictionaryMatOne}^{\ell+1}(\mathbf{G}) - f_{\DictionaryMatTwo}^{\ell+1}(\mathbf{G})}_F \\
  &\leq \underbrace{\Abs{\frac{\tau_\ell}{KN} (\mathbf{A}\DictionaryMatOne)^* \left[\left( T - S^{\mathbf{G}} \right) (\mathbf{A}\DictionaryMatTwo f_{\DictionaryMatOne}^{\ell}(\mathbf{G})) - \left( T - S^{\mathbf{G}} \right) (\mathbf{A}\DictionaryMatOne f_{\DictionaryMatOne}^{\ell}(\mathbf{G}))\right]}_F}_{\eqqcolon(*^1)} \\
  &+ \underbrace{\Abs{\frac{\tau_\ell}{KN} \left((\mathbf{A}\DictionaryMatTwo)^* - (\mathbf{A}\DictionaryMatOne)^*\right) (T-S^{\mathbf{G}}) (\mathbf{A}\DictionaryMatTwo (f_{\DictionaryMatOne}^{\ell}(\mathbf{G})))}_F}_{\eqqcolon(*^2)}\\
  & + \underbrace{\Abs{\left(f_{\DictionaryMatOne}^{\ell}(\mathbf{G}) - f_{\DictionaryMatTwo}^{\ell}(\mathbf{G})\right) - \frac{\tau_\ell}{KN} (\mathbf{A}\DictionaryMatTwo)^* \left[\left( T - S^{\mathbf{G}} \right) (\mathbf{A}\DictionaryMatTwo f_{\DictionaryMatOne}^{\ell}(\mathbf{G})) - \left( T - S^{\mathbf{G}} \right) (\mathbf{A}\DictionaryMatTwo f_{\DictionaryMatTwo}^{\ell}(\mathbf{G}))\right]}_F}_{\eqqcolon(*^3)}.
\end{align*} 
 The third term can be bounded as 
  \begin{align*}
  (*^3)&=\Abs{\left(f_{\DictionaryMatOne}^{\ell}(\mathbf{G}) - f_{\DictionaryMatTwo}^{\ell}(\mathbf{G})\right) - \tau_\ell \left[\left( T_{\DictionaryMatTwo} - S^{\mathbf{G}}_{\DictionaryMatTwo} \right) (f_{\DictionaryMatOne}^{\ell}(\mathbf{G})) - \left( T_{\DictionaryMatTwo} - S^{\mathbf{G}}_{\DictionaryMatTwo} \right) (f_{\DictionaryMatTwo}^{\ell}(\mathbf{G}))\right]}_F\\
    ~^\eqref{expr0}&\leq\left(1 + \frac{\|\mathbf{G}\|_\infty}{\delta} \right)\Abs{(f_{\DictionaryMatOne}^{\ell}(\mathbf{G})) - (f_{\DictionaryMatTwo}^{\ell}(\mathbf{G}))}_F.
  \end{align*}
  Since proximal operator are firmly nonexpansive and thus also 1-Lipschitz, we get that $T - S^{\G}$ is $\left(1 + \frac{\|\mathbf{G}\|_\infty}{\delta}\right)$-Lipschitz and we can bound the first term by:
  \begin{align*}
  (*^1)&\leq \left(1 + \frac{\|\mathbf{G}\|_\infty}{\delta}\right) \cdot \Abs{\frac{\tau_\ell}{KN} (\mathbf{A}\DictionaryMatOne)^*}_{2\to 2} \cdot\Abs{\mathbf{A}\DictionaryMatTwo f_{\DictionaryMatOne}^{\ell}(\mathbf{G}) - \mathbf{A}\DictionaryMatOne f_{\DictionaryMatOne}^{\ell}(\mathbf{G}) }_F\\
  &\leq \left(1 + \frac{\|\mathbf{G}\|_\infty}{\delta}\right) \frac{\tau_\ell}{KN} \Abs{\mathbf{A}\DictionaryMatOne}_{2\to 2} \cdot \Abs{ \mathbf{A}\DictionaryMatTwo  - \mathbf{A}\DictionaryMatOne}_{2\to 2} \Abs{f_{\DictionaryMatOne}^{\ell}(\mathbf{G}) }_F
  \end{align*}
  Using $(T-S^{\mathbf{G}})$ is $\left(1 + \frac{\|\mathbf{G}\|_\infty}{\delta}\right)$-Lipschitz  together with $T(0) = 0 = S^{\mathbf{G}}(0)$ (\cref{tandsvanish}), we obtain
  \begin{align*}
  (*^2)&\leq\Abs{\frac{\tau_\ell}{KN} \left( (\mathbf{A}\DictionaryMatTwo)^* - (\mathbf{A}\DictionaryMatOne)^* \right)}_{2\to 2}\cdot\Abs{(T-S^{\mathbf{G}}) (\mathbf{A}\DictionaryMatTwo (f_{\DictionaryMatOne}^{\ell}(\mathbf{G})))}_F\\
    &\leq \left(1 + \frac{\|\mathbf{G}\|_\infty}{\delta}\right) \Abs{\frac{\tau_\ell}{KN} \left( (\mathbf{A}\DictionaryMatTwo)^* - (\mathbf{A}\DictionaryMatOne)^* \right)}_{2\to 2}\cdot \|\mathbf{A\DictionaryMatTwo}\|_{2 \to 2}\cdot \Abs{ f_{\DictionaryMatOne}^{\ell}(\mathbf{G}) - 0}_F
  \end{align*}
  Noting $\Abs{\mathbf{A}\DictionaryMatOne}_{2\to 2}=\|\mathbf{A}\|_{2 \to 2}=\|\mathbf{A\DictionaryMatTwo}\|_{2 \to 2}$, the bounds for ($*^1$) and ($*^2$) are the same and we get
  \begin{equation*}
  (*^1)+(*^2)+(*^3)\leq\frac{2\tau_\ell}{KN} \Abs{\mathbf{A}}_{2 \to 2} \left(1 + \frac{\Abs{\mathbf{G}}_\infty}{\delta}\right)\Abs{\mathbf{A}\DictionaryMatTwo - \mathbf{A}\DictionaryMatOne}_{2 \to 2}  \Abs{ (f_{\DictionaryMatOne}^{\ell}(\mathbf{G})) }_F
  \end{equation*}
  \[
  + \left(1 + \frac{\|\mathbf{G}\|_\infty}{\delta} \right)\Abs{ (f_{\DictionaryMatOne}^{\ell}(\mathbf{G})) - (f_{\DictionaryMatTwo}^{\ell}(\mathbf{G})) }_F.
  \]
  Factoring out $\left(1 + \frac{\|\mathbf{G}\|_\infty}{\delta}\right)$ shows the desired inequality.
  \end{proof}

\begin{theorem}\label{MainTheorem: Complex}
  Let $\mathbf{A} \in \mathbb{C}^{KN \times \numpixel}$ and $L \in \mathbb{N}_0$. Furthermore, for $\ell \in \{0,\ldots L-1\}$, let $\tau_\ell \in \mathbb{R}_{>0}$ satisfy \cref{assum1}.
  Define the constants 
  \[
    \gamma := 1 + \frac{\Abs{\mathbf{G}}_\infty}{\delta}\quad \text{and}\quad \mathcal{T}_L\coloneqq\displaystyle\sum_{k=0}^{L-1}\tau_k.
  \]
  Then, for all $\DictionaryMatOne, \DictionaryMatTwo \in \mathcal{U}(N)$, the following perturbation bound holds:
  \begin{equation}\label{perturbation}
    \Abs{f_{\DictionaryMatOne}^{L}(\mathbf{G}) - f_{\DictionaryMatTwo}^{L}(\mathbf{G})}_F
    \leq K_L \Abs{\mathbf{A}\DictionaryMatOne - \mathbf{A}\DictionaryMatTwo}_{2 \to 2},
  \end{equation}
  where the constant $K_L \in (0, \infty)$ is given by
  \begin{equation}\label{klbd}
    K_L =\sum_{\ell=0}^{L-1}\gamma^{L-\ell}B_\ell\text{, where }B_\ell\coloneqq\frac{2\tau_\ell}{KN} \Abs{\mathbf{A}}_{2 \to 2} \cdot \brac{\sqrt{m} + \frac{\mathcal{T}_\ell}{KN} \Abs{\mathbf{A}}_{2 \to 2}\cdot \Abs{\G}_\infty}.
  \end{equation}
  
\end{theorem}
  \begin{proof}
    Using first \cref{perturb} and then \cref{defG}, we get for $\ell\in\mathbb{N}_0$
    \begin{align}
      &\Abs{f_{\DictionaryMatOne}^{\ell+1}(\mathbf{G}) - f_{\DictionaryMatTwo}^{\ell+1}(\mathbf{G})}_F 
    \leq \gamma \left(
      \Abs{f_{\DictionaryMatOne}^{\ell}(\mathbf{G}) - f_{\DictionaryMatTwo}^{\ell}(\mathbf{G})}_F 
    + \frac{2\tau_\ell}{KN} \Abs{\mathbf{A}}_{2 \to 2} \Abs{\mathbf{A}\DictionaryMatOne - \mathbf{A}\DictionaryMatTwo}_{2 \to 2} \cdot \Abs{f_{\DictionaryMatOne}^{\ell}(\mathbf{G})}_F \right)\notag\\
    &\leq \gamma \left(
      \Abs{f_{\DictionaryMatOne}^{\ell}(\mathbf{G}) - f_{\DictionaryMatTwo}^{\ell}(\mathbf{G})}_F 
    + \frac{2\tau_\ell}{KN} \Abs{\mathbf{A}}_{2 \to 2} \Abs{\mathbf{A}\DictionaryMatOne - \mathbf{A}\DictionaryMatTwo}_{2 \to 2} \cdot \brac{\sqrt{m} + \frac{\mathcal{T}_\ell}{KN} \Abs{\mathbf{A}}_{2 \to 2}\cdot \Abs{\G}_\infty} \right)\notag\\
    &=\gamma 
      \Abs{f_{\DictionaryMatOne}^{\ell}(\mathbf{G}) - f_{\DictionaryMatTwo}^{\ell}(\mathbf{G})}_F 
    + \gamma B_\ell \Abs{\mathbf{A}\DictionaryMatOne - \mathbf{A}\DictionaryMatTwo}_{2 \to 2}\label{eq:PerturbStep}
    \end{align}
    We proceed by induction on $L \in \mathbb{N}_0$. 
    For $L = 1$, using \eqref{eq:PerturbStep} with $\ell=L-1=0$, we have
    \begin{align*}
      \Abs{f_{\DictionaryMatOne}^{1}(\mathbf{G}) - f_{\DictionaryMatTwo}^{1}(\mathbf{G})}_F 
    &\leq \gamma \Abs{f_{\DictionaryMatOne}^{0}(\mathbf{G}) - f_{\DictionaryMatTwo}^{0}(\mathbf{G})}_F + \gamma B_0 \Abs{\mathbf{A}\DictionaryMatOne - \mathbf{A}\DictionaryMatTwo}_{2 \to 2}\\
    &=\gamma B_0 \Abs{\mathbf{A}\DictionaryMatOne - \mathbf{A}\DictionaryMatTwo}_{2 \to 2}.
    \end{align*}
    Here, we have used $f^0_{\DictionaryMatOne}(\mathbf{G}) = \mathbf{f} = f_{\DictionaryMatTwo}^{1}(\mathbf{G})$.
    
    Since $\mathcal{T}_1=\tau_0$ and $K_1 = \gamma B_0$, this shows the statement for $L=1$.
    
    Assume now that the inequality holds for some $ L \in \mathbb{N}_0 $, i.e.,
    \[
      \Abs{f_{\DictionaryMatOne}^{L}(\mathbf{G}) - f_{\DictionaryMatTwo}^{L}(\mathbf{G})}_F 
    \leq \sum_{\ell=0}^{L-1}\gamma^{L-\ell}B_\ell \Abs{\mathbf{A}\DictionaryMatOne - \mathbf{A}\DictionaryMatTwo}_{2 \to 2}.
    \]
    Using \eqref{eq:PerturbStep} with $\ell=L=0$ and then the inductive hypothesis, we have
    \begin{align*}
      \Abs{f_{\DictionaryMatOne}^{L+1}(\mathbf{G}) - f_{\DictionaryMatTwo}^{L+1}(\mathbf{G})}_F 
    &\leq\gamma 
      \Abs{f_{\DictionaryMatOne}^{L}(\mathbf{G}) - f_{\DictionaryMatTwo}^{L}(\mathbf{G})}_F 
    + \gamma B_L \Abs{\mathbf{A}\DictionaryMatOne - \mathbf{A}\DictionaryMatTwo}_{2 \to 2}\\
&\leq\gamma 
       \sum_{\ell=0}^{L-1}\gamma^{L-\ell}B_\ell \Abs{\mathbf{A}\DictionaryMatOne - \mathbf{A}\DictionaryMatTwo}_{2 \to 2} 
    + \gamma B_L \Abs{\mathbf{A}\DictionaryMatOne - \mathbf{A}\DictionaryMatTwo}_{2 \to 2}\\
    &= \brac{
       \sum_{\ell=0}^{L-1}\gamma^{L+1-\ell}B_\ell + \gamma B_L }\Abs{\mathbf{A}\DictionaryMatOne - \mathbf{A}\DictionaryMatTwo}_{2 \to 2}\\
    &= \brac{
       \sum_{\ell=0}^{L}\gamma^{L+1-\ell}B_\ell}\Abs{\mathbf{A}\DictionaryMatOne - \mathbf{A}\DictionaryMatTwo}_{2 \to 2}\qedhere
    \end{align*}
    \end{proof}
\section{Generalization Error Bound via Covering Numbers}\label{sec:gen-error-bound}

In this section, we derive a generalization error bound for the hypothesis class $\mathcal{H}_2^L$, closely following the technique from \cite{BeRaSc22}. The approach relies on bounding the covering numbers of
\begin{equation}\label{def:M_2}
    \mathcal{M}_2 \coloneqq \big\{ \sigma \big(\mathbf{\Psi} f_\mphi^L(\mathbf{G})\big) \in \mathbb{C}^{K\numpixel \times \numsample}: \mathbf{\Psi, \Phi} \in \mathcal{U}(\numpixel) \big\}.
    \end{equation}
which represents the family of network outputs indexed by unitary parameters. We begin by quantifying how changes in the parameters affect the outputs in $\mathcal{M}_2$.\\

The following corollary provides a Lipschitz-type inequality showing that the Frobenius norm of output differences in $\mathcal{M}_2$ can be bounded in terms of the distances between parameter matrices:

\begin{corollary}\label{Cordis}
  Let $L \geq 1$, and assume that \cref{assum1} holds for all $\tau_\ell$ with $\ell = 0, \ldots, L-1$. 
Moreover let $\mathcal{R} \in \Gamma_0(\mathbb{C})$ with $0 \in \argmin_{x \in \mathbb{R}^n} \mathcal{R}(x)$
Then, for any $\DictionaryPsiOne, \DictionaryPsiTwo \in \mathcal{U}(\numpixel)$ and $\DictionaryMatOne, \DictionaryMatTwo \in \mathcal{U}(\numpixel)$, the following estimate holds:
\[
\left\|\sigma\brac{\DictionaryPsiOne f_{\DictionaryMatOne}^L(\mathbf{G})}
   -\sigma\brac{\DictionaryPsiTwo f_{\DictionaryMatTwo}^L(\mathbf{G})}\right\|_F 
   \leq M_L \left\|\DictionaryPsiOne - \DictionaryPsiTwo\right\|_{2 \rightarrow 2}
   + M'_L \left\|\mathbf{A}\DictionaryMatOne - \mathbf{A}\DictionaryMatTwo\right\|_{2 \to 2}.
\]
  with
  \begin{equation}\label{mlbd}
  M_L=\sqrt{m} + \frac{\mathcal{T}_L}{KN} \Abs{\mathbf{A}}_{2 \to 2} \cdot \Abs{\G}_\infty,\quad M'_L = K_L
  \end{equation}
  with $K_L$ and $\mathcal{T}_L$ from \cref{MainTheorem: Complex}.
  \end{corollary}
  \begin{proof}
  By applying the 1-Lipschitz continuity of the function $\sigma$ (see \eqref{sigmacases}), 
together with the triangle inequality, and using that 
$\DictionaryPsiOne, \DictionaryPsiTwo \in \mathcal{U}(\numpixel)$, 
as well as \eqref{out1} and \eqref{perturbation}, we obtain

  \begin{equation*}
  \begin{split}
  &\left\|\sigma\brac{\DictionaryPsiOne f_{\DictionaryMatOne}^L(\mathbf{G})}-\sigma\brac{\DictionaryPsiTwo f_{\DictionaryMatTwo}^L(\mathbf{G})}\right\|_F 
  \leq\left\|\DictionaryPsiOne f_{\DictionaryMatOne}^L(\mathbf{G})-\DictionaryPsiTwo f_{\DictionaryMatTwo}^L(\mathbf{G})\right\|_F \\
  & \leq\left\|\DictionaryPsiOne f_{\DictionaryMatOne}^L(\mathbf{G})-\DictionaryPsiTwo f_{\DictionaryMatOne}^L(\mathbf{G})\right\|_F+\left\|\DictionaryPsiTwo f_{\DictionaryMatOne}^L(\mathbf{G})-\DictionaryPsiTwo f_{\DictionaryMatTwo}^L(\mathbf{G})\right\|_F \\
  &=\left\|\brac{\DictionaryPsiOne-\DictionaryPsiTwo} f_{\DictionaryMatOne}^L(\mathbf{G})\right\|_F+\left\|\DictionaryPsiTwo\brac{f_{\DictionaryMatOne}^L(\mathbf{G})-f_{\DictionaryMatTwo}^L(\mathbf{G})}\right\|_F \\
  & \leq\left\|\DictionaryPsiOne-\DictionaryPsiTwo\right\|_{2 \rightarrow 2}\left\|f_{\DictionaryMatOne}^L(\mathbf{G})\right\|_F+\left\|f_{\DictionaryMatOne}^L(\mathbf{G})-f_{\DictionaryMatTwo}^L(\mathbf{G})\right\|_F \\
  & \leq\brac{\sqrt{m} + \frac{\mathcal{T}_L}{KN} \Abs{\mathbf{A}}_{2 \to 2} \cdot \Abs{\G}_\infty}\left\|\DictionaryPsiOne-\DictionaryPsiTwo\right\|_{2 \rightarrow 2} + K_L \Abs{\mathbf{A}\DictionaryMatOne - \mathbf{A}\DictionaryMatTwo}_{2 \to 2}.\qedhere
  \end{split}
  \end{equation*}
  \end{proof}
 Using the inequality from Corollary~\ref{Cordis}, we can define a metric on the parameter space $\mathcal{U}(\numpixel) \times \mathcal{U}(\numpixel)$ by appropriately scaling the operator norms. %
Let
\[
d_1(\DictionaryPsiOne,\DictionaryPsiTwo) \coloneqq M_L\left\|\DictionaryPsiOne-\DictionaryPsiTwo\right\|_{2 \rightarrow 2},
\quad
d_2(\DictionaryMatOne,\DictionaryMatTwo) \coloneqq M'_L\Abs{\mathbf{A}\DictionaryMatOne - \mathbf{A}\DictionaryMatTwo}_{2 \to 2}.
\]
Since $\left\|.\right\|_{2 \to 2}$ is a norm and $M_L, M'_L \geq 0$, the function
\[
d\big((\DictionaryMatOne, \DictionaryPsiOne), (\DictionaryMatTwo, \DictionaryPsiTwo)\big)
\coloneqq d_1(\DictionaryPsiOne,\DictionaryPsiTwo) + d_2 (\DictionaryMatOne,\DictionaryMatTwo)
\]
defines a metric on $\mathcal{U}(\numpixel) \times \mathcal{U}(\numpixel)$. Therefore, $(\mathcal{U}(\numpixel) \times \mathcal{U}(\numpixel), d)$ is a metric space. Moreover, the Frobenius norm between any two elements of $\mathcal{M}_2$ can be bounded above by this metric:
\begin{equation}\label{dist}
  \begin{split}
 \left\|m_1 - m_2\right\|_F & = \left\|\sigma\brac{\DictionaryPsiOne f_{\mathbf{\Phi}_1}^L(\mathbf{G})}-\sigma\brac{\DictionaryPsiTwo f_{\mathbf{\Phi}_2}^L(\mathbf{G})}\right\|_F \\ &\leq M_L\left\|\DictionaryPsiOne-\DictionaryPsiTwo\right\|_{2 \rightarrow 2} + M'_L \Abs{\mathbf{A}\DictionaryMatOne - \mathbf{A}\DictionaryMatTwo}_{2 \to 2}\\ &= d_1(\DictionaryPsiOne\;,\;\DictionaryPsiTwo) + d_2(\DictionaryMatOne\;,\;\DictionaryMatTwo)
 = d((\DictionaryMatOne, \DictionaryPsiOne), (\DictionaryMatTwo, \DictionaryPsiTwo))
  \end{split}
\end{equation}
In this sense, the distance on $\mathcal{M}_2$ can be controlled by the metric $d$ on $\mathcal{U}(\numpixel)\times \mathcal{U}(\numpixel)$.

From this, we immediately obtain that the covering numbers of $\mathcal{M}_2$ under the Frobenius norm can be upper bounded by those of its parameter space:
\begin{equation}\label{rem60}
  \mathcal{N}\left(\mathcal{M}_2,\; \lVert \cdot \rVert_{F},\; \varepsilon\right) 
  \leq \mathcal{N}\left(\mathcal{U}(\numpixel) \times \mathcal{U}(\numpixel),\; d,\; \varepsilon\right).
\end{equation}
To see this, let $n:= \mathcal{N}\brac{\mathcal{U}(\numpixel)\times \mathcal{U}(\numpixel)\;,\; d\;,\; \varepsilon}$. Then, there exists $\big(\mathbf{\Psi}_i\;,\;\mphi_i\big)_{i = 1}^{n}$ such that
\[
\mathcal{U}(\numpixel)\times \mathcal{U}(\numpixel) \subset \bigcup_{i = 1}^{n}\mathcal{B}^{d}_{\epsilon}\big((\mathbf{\Psi}_i\;,\;\mphi_i)\big).
\]
Let $m \in M_2$. Then, there exists $\big(\mathbf{\Psi}\;,\;\mphi\big) \in \mathcal{U}(\numpixel)\times \mathcal{U}(\numpixel)$ such that $m = \sigma\brac{\mathbf{\Psi} f_{\mphi}^L(\mathbf{G})}$. Furthermore, there exists a $j\in\{1,\ldots,n\}$ with 
\[
d\big((\mathbf{\Psi}, \mathbf{\Psi}_j),(\mphi, \mphi_j)\big)\leq \epsilon.
\]
Now, using \eqref{dist}, we obtain
\[
\|m - m_j\|_F = \|\sigma\brac{\mathbf{\Psi} f_{\mphi}^L(\mathbf{G})} - \sigma\brac{\mathbf{\Psi}_j f_{\mphi_j}^L(\mathbf{G})}\| 
\leq d\big((\mathbf{\Psi}, \mathbf{\Psi}_j),(\mphi, \mphi_j)\big) \leq \epsilon.
\] 
Thus, $m \in \bigcup_{i = 1}^{n}\mathcal{B}^{\left\|.\right\|_F}_{\epsilon}(m_i)$ where $m_i = \sigma\brac{\mathbf{\Psi}_i f_{\mphi_i}^L(\mathbf{G})}$. Since $m \in M_2$, was arbitrary, we have $M\subset \bigcup_{i = 1}^{n}\mathcal{B}^{\left\|.\right\|_F}_{\epsilon}(m_i)$ and consequently
\[
\mathcal{N}\bigg(\mathcal{M}_2\;,\; \lVert \cdot \rVert_{F}\;,\; \varepsilon\bigg) 
\leq n = \mathcal{N}\bigg(\mathcal{U}(\numpixel)\times \mathcal{U}(\numpixel)\;,\; d\;,\; \varepsilon\bigg).
\]

We now derive an explicit upper bound on the covering numbers of $\mathcal{M}_2$ based on the metric properties established above.

 \begin{lemma}[Covering Number Bound for \texorpdfstring{$\mathcal{M}_2$}{}]\label{lem:covering_M2}
  Let the assumptions of \cref{Cordis} be fulfilled. Then,
\begin{equation*}
  \log \left( \mathcal{N}(\mathcal{M}_2, \lVert \cdot \rVert_{2 \to 2}, \varepsilon) \right) \leq 2\numpixel^2 \log \left(1 + \frac{4 M_L}{\varepsilon} \right)  +2K \numpixel^2 \log \left(1 + \frac{4 M'_L\lVert \mathbf{A} \rVert_{2\to 2}}{\varepsilon} \right).
\end{equation*}
\end{lemma}
\begin{proof}
We argue as follows. First estimate the covering number of the set
\[
\mathcal{W}= \{\mathbf{A \Phi}: \mphi \in \mathcal{U}(\numpixel)\} = \left \{\lVert \A \rVert_{2 \to 2} \cdot \frac{\A \mphi}{\lVert \A \rVert_{2 \to 2}}: \mphi \in \mathcal{U}(\numpixel) \right\} \subset \mathbb{C}^{KN \times \numpixel}.
\]
To control $\mathcal{N}(\mathcal{W}, \lVert \cdot \rVert_{2 \to 2}, \xi)$ for $\xi>0$, we extend \cite[Lemma~6]{BeRaSc22} to the complex case. Using the identity above for $\mathcal{W}$, we get
\begin{equation}
\label{eq:WCoverNum}
\mathcal{N}(\mathcal{W}, \lVert \cdot \rVert_{2 \to 2}, \xi) 
= \mathcal{N} \left( \left\{ \frac{\A \mphi}{\lVert \A \rVert_{2 \to 2}}: \mphi \in \mathcal{U}(\numpixel) \right\}, \lVert \cdot \rVert_{2 \to 2}, \tfrac{\xi}{\lVert \A \rVert_{2 \to 2}} \right) 
\leq \left( 1 + \frac{2 \lVert \mathbf{A} \rVert_{2 \to 2}}{\xi} \right)^{2K \numpixel^2}.
\end{equation}
The last inequality follows from $\mathcal{U}(\numpixel) \subset  B^{\numpixel \times \numpixel}_{\lVert \cdot \rVert_{2 \to 2}}:=\{B\in\mathbb{C}^{\numpixel\times\numpixel}:\lVert B \rVert_{2 \to 2}\leq1\}$ and \cite[Proposition~C.3]{FoRa13}. Since the latter is formulated for the real case, we identify $\mathbb{C}^{KN \times \numpixel}$ with $\mathbb{R}^{2\times KN \times \numpixel}$ to apply this proposition. Note that this doubles the dimension of the space, hence doubling the exponent in the covering bound.

Next, using \eqref{rem60} together with the product-space covering bound, we get
\begin{align*}
\mathcal{N}(\mathcal{M}_2, \lVert \cdot \rVert_F, \varepsilon) 
&\leq \mathcal{N}(\mathcal{U}(\numpixel)\times \mathcal{U}(\numpixel), d, \varepsilon)
= \mathcal{N}(\mathcal{U}(\numpixel)\times \mathcal{U}(\numpixel), d_1 + d_2, \varepsilon)\\
&\leq \mathcal{N}(\mathcal{U}(\numpixel), d_1, \tfrac{\varepsilon}{2})
\times \mathcal{N}(\mathcal{U}(\numpixel), d_2, \tfrac{\varepsilon}{2})\\
&= \mathcal{N}\left(\mathcal{U}(\numpixel), M_L \lVert \cdot \rVert_{2 \to 2}, \tfrac{\varepsilon}{2} \right) 
\times \mathcal{N}\left( \mathcal{W}, M'_L \lVert \cdot \rVert_{2 \to 2}, \tfrac{\varepsilon}{2} \right)\\
&\leq \mathcal{N}\left(\mathcal{U}(\numpixel), \lVert \cdot \rVert_{2 \to 2}, \tfrac{\varepsilon}{2 M_L} \right) 
\times \mathcal{N}\left( \mathcal{W}, \lVert \cdot \rVert_{2 \to 2}, \tfrac{\varepsilon}{2 M'_L} \right).
\end{align*}
Since $\mathcal{U}(\numpixel) \subset \mathcal{B}^{\numpixel \times \numpixel}_{\lVert \cdot \rVert_{2 \to 2}}$, the covering number bound from \cite[Proposition~C.3]{FoRa13} (doubling the dimension, since we apply it for complex spaces) gives for $\xi>0$
\[
\mathcal{N}(\mathcal{U}(\numpixel), \lVert \cdot \rVert_{2 \to 2}, \xi) \leq \left(1 + \frac{2}{\xi} \right)^{2\numpixel^2}
\]
Substituting $\xi = \tfrac{\varepsilon}{2 M_L}$ above and $\xi = \tfrac{\varepsilon}{2 M'_L}$ in \eqref{eq:WCoverNum}, we obtain
\[
\mathcal{N}(\mathcal{M}_2, \lVert \cdot \rVert_F, \varepsilon) 
\leq \left( 1 + \frac{4 M_L}{\varepsilon} \right)^{2\numpixel^2} 
\times \left( 1 + \frac{4 M'_L \lVert \A \rVert_{2 \to 2}}{\varepsilon} \right)^{2K \numpixel^2}.
\]
Taking logarithms gives the claimed result.
\end{proof}
Using the covering number bound from Lemma~\ref{lem:covering_M2}, we now present the main generalization bound for the hypothesis class $\mathcal{H}_2^L$, based on Rademacher complexity estimates. Here, the \emph{generalization error} is defined as the difference between the true risk of a hypothesis $h \in \mathcal{H}_2^L$ on unseen data and its empirical risk measured on the training set. This quantity measures how much worse a hypothesis $h \in \mathcal{H}_2^L$ may perform on unseen data compared to its training performance (see Subsection~2.1 of \cite{BeRaSc22}).

\begin{theorem}\label{Th:Main_Result}
Consider the hypothesis space $\mathcal{H}^L_2$ in \eqref{H: spaces}. Let $\alpha \in (0,1)$. With a probability at least $1 - \alpha$, for all $h \in \mathcal{H}^L_2$, the generalization error is bounded by
\begin{equation*}
\mathcal{L}(h) - \hat{\mathcal{L}}(h) \leq  2 \mathcal{R}_m(l \circ \mathcal{H}_2^L) + 4 (C_{\text{in}} + C_{\text{out}}) \sqrt{\frac{2 \log (\sfrac{4}{\alpha})}{m}}, 
\end{equation*}
where $C_{\text{in}}$ and $C_{\text{out}}$ are the constants defined in \cref{sec: unrolling}. 
The Rademacher complexity term is further bounded by
\begin{equation*}
\mathcal{R}_m\left(l \circ \mathcal{H}_2^{L}\right) \leq 4 C_\text{out} \frac{\numpixel}{\sqrt{\numsample}} \brac{\sqrt{\log \left( e \left( 1+ \frac{8 M_L}{\sqrt{m}C_\text{out} } \right) \right) }+ \sqrt{\log \left( e \left( 1+ \frac{8 \lVert \mathbf{A} M'_L \rVert_{ 2\to 2}}{\sqrt{m}C_\text{out} } \right) \right) }}, 
\end{equation*}
where $M_L$ and $M'_L$ are as in \eqref{mlbd}.  
\end{theorem}
\begin{proof}
Noting that the loss in our case is bounded by $c = C_{\text{in}} + C_{\text{out}}$, the claimed upper bound for $\mathcal{L}(h) - \hat{\mathcal{L}}(h)$ follows from \cite[Theorem 26.5]{ShBe14}.
The bound for the Rademacher complexity is based on Dudley's inequality \cite[Theorem 8.23]{FoRa13}. Following the derivation of the bound derived in \cite[Subsection~4.3, (33)]{BeRaSc22}, one gets
\[
\mathcal{R}_m\brac{l \circ \mathcal{H}_2^{L}}
\;\leq\; \frac{4 \sqrt{2}}{\numsample} \int_0^{\sqrt{\numsample} \sfrac{C_{\text{out}}}{2}} 
\sqrt{\log \brac{\mathcal{N}\brac{\mathcal{M}_2, \lVert \cdot \rVert_{ 2 \to 2}, \varepsilon}}} 
\,d\varepsilon.
\]
Using this, together with \cref{lem:covering_M2} and $\sqrt{a+b}\leq\sqrt{a}+\sqrt{b}$ for $a,b\geq0$, we get
\begin{equation*}
  \begin{split} \mathcal{R}_m\brac{l \circ \mathcal{H}_2^{L}} &\leq \frac{4 \sqrt{2}}{\numsample} \int_0^{\sqrt{\numsample} \sfrac{C_{\text{out}}}{2}} \sqrt{\log \brac{\mathcal{N}\brac{\mathcal{M}_2, \lVert \cdot \rVert_{ 2 \to 2}, \varepsilon}}} \,d\varepsilon\\
     &\leq \frac{4 \sqrt{2}}{\numsample} \int_0^{\sqrt{\numsample} \sfrac{C_{\text{out}}}{2}} \sqrt{2\numpixel^2 \log \brac{1 + \frac{4 M_L}{\varepsilon}}} \,d\varepsilon \\
      &\quad + \frac{4 \sqrt{2}}{\numsample} \int_0^{\sqrt{\numsample} \sfrac{C_{\text{out}}}{2}} \sqrt{2K\numpixel^2 \log \brac{1 + \frac{4M'_L \Abs{\mathbf{A}}_{2\to 2}}{\varepsilon}}} \,d\varepsilon\\
       &\leq 4 C_{\text{out}} \frac{\numpixel}{\sqrt{\numsample}} \sqrt{\log \brac{e \brac{1 + \frac{4 M_L}{\sqrt{\numsample} \sfrac{C_{\text{out}}}{2}}}}} \\
        &\quad + 4 C_{\text{out}} \sqrt{K}\frac{\numpixel}{\sqrt{\numsample}} \sqrt{\log \brac{e \brac{1 + \frac{4 M'_L\Abs{\mathbf{A}}_{2\to 2}}{\sqrt{\numsample} \sfrac{C_{\text{out}}}{2}}}}},
  \end{split}
\end{equation*}
where the last inequality holds due to \cite[(47)]{BeRaSc22} (which slightly generalizes \cite[Lemma~C.9]{FoRa13}) with $\alpha=\sqrt{\numsample} \sfrac{C_{\text{out}}}{2}$ and $\beta=4 M_L$ or $\beta=4M'_L\Abs{\mathbf{A}}_{2\to 2}$.
\end{proof}

We conclude this section with a discussion of how the generalization bound scales with the network depth $L$, the parameter dimension $\numpixel^2$, the number of measurement operators $K$, and the number of training samples $m$.

\begin{remark}\label{rem:scaling}
$M_L$ grows (at most) linearly in $L$, while $M'_L$ scales exponentially in $L$. This can be seen as follows. Noting
\[
\mathcal{T}_L = \sum_{k=0}^{L-1} \tau_k \;\le\; L \max_{k \in [0, L-1]} \tau_k,
\]
$\mathcal{T}_L$ and thus $M_L$ grows (at most) linearly in $L$.
We recall
\[
K_L = \sum_{\ell=0}^{L-1} \gamma^{L-\ell}B_\ell\quad\text{with}\quad \gamma = 1 + \frac{\Abs{\mathbf{G}}_\infty}{\delta} >1,
\]
and note that $B_\ell$ depends linearly on $\mathcal{T}_\ell$. With the above, we get
$
B_\ell = \mathcal{O}(\ell).
$
Factoring out $\gamma^L$ in the expression for $K_L$ above gives
\[
K_L = \gamma^L \sum_{\ell=0}^{L-1} \gamma^{-\ell} B_\ell
= \gamma^L \sum_{\ell=0}^{L-1} \mathcal{O}\!\left( \frac{\ell}{\gamma^\ell} \right).
\]
Since $\ell/\gamma^\ell$ decays, the inner sum is bounded. Therefore,
$
K_L = \mathcal{O}(\gamma^L),
$
showing that $K_L=M'_L$ grows exponentially in $L$.

By isolating the dependence on $K,$ $\numpixel,$ $\numsample$ and $L$, and treating other terms as constants, the generalization error is bounded as
\begin{equation*}
\mathcal{L}(h) - \hat{\mathcal{L}}(h) \;\lesssim\; \frac{N}{\sqrt{m}}\sqrt{\log (L)} + \frac{N \sqrt{K}}{\sqrt{m}}\sqrt{L} \;\lesssim\; \frac{N(1 + \sqrt{K})}{\sqrt{m}} \sqrt{L}.
\end{equation*}
This reveals an overall square-root dependency on the depth $L$ (since the logarithmic terms are dominated by $\sqrt{L}$). Moreover, the dependency on the number of trainable parameters can be made explicit: since the dimension of the parameter space is $\dim \mathcal{U}(N) = N^2 = \mathcal{O}(N^2)$, the bound can equivalently be expressed in terms of the ratio of trainable parameters to samples as
\[
\mathcal{L}(h) - \hat{\mathcal{L}}(h) \;\lesssim\; \sqrt{\tfrac{P}{m}} (1+\sqrt{K})\sqrt{L}, \quad P \sim N^2.
\]
Hence, the error decreases at the standard rate $1/\sqrt{m}$ with the number of samples, while growing with the square root of the ratio $P/m$, the depth $L$, and the measurement count $K$.
This shows that any factor increase in $N^2$ must be compensated by the same factor increase in $m$ to preserve generalization performance. Hence, the number of samples must scale proportionally to the number of trainable parameters to preserve generalization performance.
\end{remark}
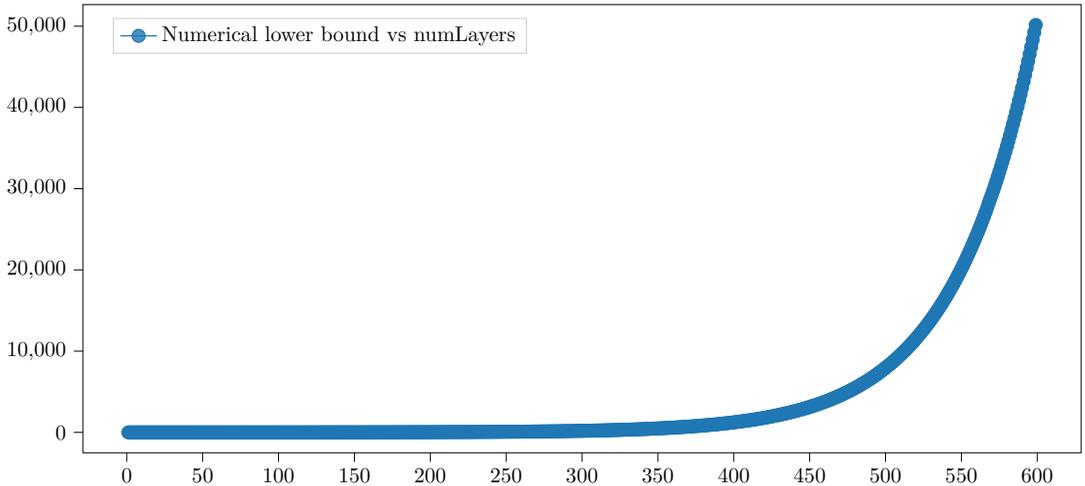
\begin{figure}[H]
  \centering
  \resizebox{\linewidth}{!}{\input{LipschitzGrowth.tex}}
  \caption{Lower bound for $K_L$ computed by numerically maximizing the ratio $\frac{\lVert f_{\DictionaryMat_1}^{L}(\G) - f_{\DictionaryMat_2}^{L}(\G)\rVert_F}{\lVert {\DictionaryMat}_1 - {\DictionaryMat}_2 \rVert_{2 \to 2}}$ with respect to ${\DictionaryMat}_1$ and ${\DictionaryMat}_2$ for the case $N=2$, $K=1$. Here, the dictionary matrices were restricted to be rotations and the optimization was done over the rotation angles.}
  \label{fig:lipschitz-growth}
\end{figure}
As we have seen, the resulting generalization bound is linear in the number of layers $L$ (unlike the bound of \cite{BeRaSc22}, which is logarithmic in $L$). This deficiency comes from the pertubation bound, where $K_L$ grows exponentially in $L$ with our proof, not quadratically like in \cite{BeRaSc22}. Thus, the question arises whether this is an artifact of our proof or a inevitable consequence of the nonlinear forward model.
Figure~\ref{fig:lipschitz-growth} provides numerical evidence supporting the scaling with $L$ found in \cref{MainTheorem: Complex}. Without regularization ($\mathcal{R} = 0$), the numerically computed lower bound for the Lipschitz constant for network perturbations appears to grow exponentially with the network depth $L$, in agreement with the bound $K_L = \mathcal{O}(\gamma^L)$, where $\gamma = 1 + \Abs{G}_\infty / \delta$. 

Using $\mathcal{R}=\lambda\Abs{\cdot}_1$ and strong regularization (i.e., $\lambda$ large), it may be possible to get a logarithmic bound. The problem is that the forward step of the PGA is expanding in the vicinity of the origin.
This is due to $\varphi_\delta'$ having a large local Lipschitz constant $\sim 1/\delta$ only near the origin.
Since the backward step of the PGA with $\mathcal{R}=\lambda\Abs{\cdot}_1$ is soft thresholding, this could possibly be used to compensate for the behavior of the forward step near the origin.

Finally, we would like to point out that the generalization bound from \cite{BeRaSc22} for LISTA also follows from our approach. By choosing $\varphi(z)=z$, $\mathcal{R}=\lambda\Abs{\cdot}_1$ and restricting to the real case, \eqref{eq:EWObjective} is the objective from ISTA. In this case, one gets $\varphi'(z)=1$ and $\varphi\varphi'=\operatorname{Id}$, where $\operatorname{Id}$ denotes the identity mapping. Moreover, $\varphi\varphi'=\operatorname{Id}=\prox_0$ (the proximal operator of the constant zero mapping) and $\varphi'=\prox_{I_{\{1\}}}$ (the proximal operator of indicator function $I_{\{1\}}(z)=0$ for $z=1$ and $\infty$ else.). Since the Lipschitz constant of $\varphi'\equiv1$ is 0, one gets $\gamma=1$ in \cref{MainTheorem: Complex} which restores the quadratic growth of the pertubation bound in $L$ and thus the logarithmic scaling in $L$ of the generalization bound.

\section*{Acknowledgments}
Moussa Atwi and Benjamin Berkels were supported by the German research foundation (DFG) within the Collaborative Research Centre SFB 1481 ``Sparsity and Singular Structures'' (Project ID 442047500, Project A08).

\bibliographystyle{plain}
\bibliography{references.bib}

\end{document}

%% file: LipschitzGrowth.tex
\pgfplotsset{scaled y ticks=false}
\begin{tikzpicture}

\definecolor{darkgray176}{RGB}{176,176,176}
\definecolor{lightgray204}{RGB}{204,204,204}
\definecolor{steelblue31119180}{RGB}{31,119,180}

\begin{axis}[
legend cell align={left},
legend style={
  fill opacity=0.8,
  draw opacity=1,
  text opacity=1,
  at={(0.03,0.97)},
  anchor=north west,
  draw=lightgray204
},
tick align=outside,
tick pos=left,
x grid style={darkgray176},
xmin=-28.9, xmax=628.9,
xtick style={color=black},
y grid style={darkgray176},
ymin=-2507.78107459048, ymax=52664.0098338137,
ytick style={color=black},
width = 18cm,
height = 9cm
]
\addplot [semithick, steelblue31119180, mark=*, mark size=3, mark options={solid}]
table {%
1 0.0276030642565206
2 0.0417778454407734
3 0.0562078597958907
4 0.070898018830631
5 0.0858533250243657
6 0.101078873551406
7 0.116579854076395
8 0.132361552460448
9 0.148429352698736
10 0.164788738674567
11 0.181445296141595
12 0.198404714539093
13 0.215672789093169
14 0.233255422708099
15 0.251158628055986
16 0.269388529649912
17 0.287951365971659
18 0.306853491612616
19 0.326101379426315
20 0.345701622856155
21 0.365660938140816
22 0.385986166655237
23 0.4066842773111
24 0.427762368897806
25 0.449227672567276
26 0.471087554379936
27 0.493349517795326
28 0.516021206281737
29 0.539110405959207
30 0.56262504825977
31 0.586573212755866
32 0.610963129870041
33 0.635803183757882
34 0.66110191521927
35 0.686868024603331
36 0.71311037488346
37 0.739837994691378
38 0.767060081403166
39 0.794786004400063
40 0.823025308217478
41 0.851787715963232
42 0.881083132575212
43 0.910921648274252
44 0.941313542100491
45 0.972269285451261
46 1.00379954564262
47 1.03591518973225
48 1.06862728816131
49 1.1019471185647
50 1.13588616983172
51 1.17045614595416
52 1.2056689700917
53 1.24153678873011
54 1.27807197587263
55 1.31528713731319
56 1.35319511504498
57 1.39180899166748
58 1.43114209486283
59 1.47120800211458
60 1.51202054532888
61 1.55359381572959
62 1.59594216860228
63 1.63908022832634
64 1.68302289346065
65 1.72778534187483
66 1.773383035994
67 1.81983172821344
68 1.86714746625978
69 1.91534659881515
70 1.96444578117439
71 2.01446198097461
72 2.06541248412956
73 2.11731490076706
74 2.17018717132863
75 2.22404757284231
76 2.27891472516242
77 2.33480759748649
78 2.39174551489564
79 2.44974816505734
80 2.50883560505874
81 2.56902826832973
82 2.63034697171408
83 2.69281292270739
84 2.75644772680064
85 2.82127339494801
86 2.88731235116233
87 2.95458744037183
88 3.02312193625644
89 3.09293954934551
90 3.16406443519574
91 3.23652120278017
92 3.31033492305031
93 3.38553113753767
94 3.46213586730483
95 3.54017562183665
96 3.61967740835818
97 3.70066874111009
98 3.783177650887
99 3.86723269473642
100 3.95286296590177
101 4.04009810387085
102 4.1289683046479
103 4.21950433119848
104 4.31173752414232
105 4.40569981258599
106 4.50142372519098
107 4.598942401416
108 4.69828960306651
109 4.7994997259681
110 4.90260781180535
111 5.00764956036391
112 5.11466134184714
113 5.2236802095249
114 5.33474391253124
115 5.44789090891997
116 5.5631603790508
117 5.68059223910621
118 5.80022715499838
119 5.92210655635386
120 6.04627265099093
121 6.17276843941176
122 6.30163772981546
123 6.43292515323663
124 6.56667617899814
125 6.70293713047301
126 6.84175520111414
127 6.98317847086588
128 7.12725592274762
129 7.27403745989692
130 7.42357392286061
131 7.57591710710419
132 7.73111978110311
133 7.88923570453487
134 8.05031964694402
135 8.21442740669811
136 8.38161583035191
137 8.55194283232419
138 8.72546741496297
139 8.9022496890323
140 9.08235089444879
141 9.26583342161149
142 9.45276083296835
143 9.64319788503223
144 9.83721055084666
145 10.0348660428235
146 10.2362328360239
147 10.4413806919646
148 10.6503806827021
149 10.8633052155025
150 11.0802280578893
151 11.301224363305
152 11.5263706970402
153 11.7557450628076
154 11.9894269298354
155 12.2274972603585
156 12.4700385375781
157 12.7171347943967
158 12.9688716424713
159 13.2253363018731
160 13.4866176313379
161 13.7528061590293
162 14.0239941139422
163 14.3002754578179
164 14.581745917766
165 14.8685030193867
166 15.160646120601
167 15.458276446049
168 15.7614971221937
169 16.0704132130414
170 16.385131756583
171 16.7057618018043
172 17.0324144465883
173 17.3652028761328
174 17.7042424022326
175 18.0496505031857
176 18.4015468645083
177 18.7600534204562
178 19.1252943962268
179 19.4973963510866
180 19.8764882221042
181 20.262701368933
182 20.6561696193314
183 21.0570293154462
184 21.4654193611863
185 21.8814812702671
186 22.3053592152948
187 22.737200077697
188 23.177153498638
189 23.6253719309425
190 24.0820106917467
191 24.5472280165236
192 25.0211851138024
193 25.5040462210714
194 25.9959786616618
195 26.4971529027767
196 27.0077426145253
197 27.5279247301168
198 28.0578795072133
199 28.5977905903161
200 29.1478450745061
201 29.7082335701462
202 30.2791502690996
203 30.8607930119305
204 31.4533633563939
205 32.0570666474367
206 32.6721120882455
207 33.2987128127767
208 33.9370859596396
209 34.5874527472983
210 35.2500385507487
211 35.9250729796268
212 36.6127899577987
213 37.3134278043571
214 38.0272293162463
215 38.7544418524335
216 39.4953174195506
217 40.2501127592858
218 41.0190894372827
219 41.8025139338927
220 42.6006577363986
221 43.4138037414766
222 44.2422214800894
223 45.0862039988339
224 45.9460437727268
225 46.8220387722894
226 47.7144925667035
227 48.6237144290358
228 49.5500194434288
229 50.4937286143112
230 51.4551689776081
231 52.4346737140336
232 53.4325822646088
233 54.4492404482565
234 55.4850005815982
235 56.5402216011649
236 57.6152691876211
237 58.710515892581
238 59.826341267681
239 60.9631319961017
240 62.1212820265573
241 63.3011927097882
242 64.5032729376731
243 65.7279392849149
244 66.9756161533357
245 68.2467359189865
246 69.5417390819905
247 70.8610744190854
248 72.2051991392375
249 73.5745790419842
250 74.9696886789123
251 76.3910115180093
252 77.8390401112523
253 79.3142762651851
254 80.8172312149386
255 82.3484258012264
256 83.9083906508675
257 85.4976663606557
258 87.1168036846711
259 88.7663637250804
260 90.4469181265484
261 92.1590492743172
262 93.9033504959853
263 95.6804262670344
264 97.4908924203495
265 99.335376359511
266 101.21451727619
267 103.128966371556
268 105.079387081963
269 107.066455308788
270 109.090859652528
271 111.153484209826
272 113.254689068589
273 115.395375051096
274 117.57628401989
275 119.79817177712
276 122.061808326396
277 124.367978139709
278 126.717480429275
279 129.111129424502
280 131.549754654134
281 134.034201233694
282 136.565330158481
283 139.144018601832
284 141.77116021917
285 144.447665457693
286 147.174461871856
287 149.95249444485
288 152.782725916039
289 155.666137114661
290 158.603727299679
291 161.596514506132
292 164.645535897842
293 167.751848126938
294 170.916527699988
295 174.140671351048
296 177.425396421735
297 180.771841248501
298 184.181165557019
299 187.65455086415
300 191.193200887389
301 194.798341961947
302 198.471223465848
303 202.213118252728
304 206.025323093095
305 209.909159123623
306 213.865972304987
307 217.89713388837
308 222.004040890519
309 226.188116578016
310 230.45081096041
311 234.793601292756
312 239.217992587421
313 243.725518135802
314 248.317740039551
315 252.996249751941
316 257.762668629336
317 262.618648493061
318 267.565872201815
319 272.606054234885
320 277.740941286172
321 282.972312869557
322 288.301981935485
323 293.731795499326
324 299.263635281276
325 304.899418358484
326 310.641097829374
327 316.490663490498
328 322.450142525956
329 328.521600209863
330 334.707140622072
331 341.008907377268
332 347.429084367729
333 353.969896520126
334 360.633610566423
335 367.422535829463
336 374.339025023021
337 381.385475067163
338 388.564327918724
339 395.878071417507
340 403.329240148302
341 410.920550876627
342 418.65437294243
343 426.533513773811
344 434.560703909997
345 442.738725195972
346 451.070411746454
347 459.558650927969
348 468.206384359308
349 477.016608930914
350 485.992377843306
351 495.136801665029
352 504.453049410528
353 513.944349638214
354 523.613991569187
355 533.46532622695
356 543.501767598554
357 553.726793817554
358 564.143948369136
359 574.756841317939
360 585.569150558893
361 596.584623091594
362 607.807076318599
363 619.240399368061
364 630.888554441253
365 642.755578185417
366 654.845583092313
367 667.162758923173
368 679.711374160176
369 692.495777485488
370 705.520399287837
371 718.789753197487
372 732.308437649976
373 746.081137479266
374 760.112625540816
375 774.407764365091
376 788.971507842146
377 803.808902937774
378 818.925091441972
379 834.325311750168
380 850.014900677932
381 865.999295309665
382 882.28403488212
383 898.874762703186
384 915.777228106732
385 932.997288444122
386 950.540911113066
387 968.414175624684
388 986.625884831991
389 1005.17728043615
390 1024.0772589593
391 1043.33237031724
392 1062.94928749875
393 1082.93480887833
394 1103.29594468613
395 1124.03958786155
396 1145.17300674006
397 1166.70352526017
398 1188.63389646511
399 1210.98086864061
400 1233.74773284253
401 1256.94794233832
402 1280.57872593974
403 1304.65353500132
404 1329.18071282167
405 1354.16875947202
406 1379.62633474085
407 1405.56226113444
408 1431.98552693493
409 1458.90528931564
410 1486.33087771016
411 1514.27179651319
412 1542.73772881416
413 1571.73853968684
414 1601.28427957115
415 1631.38518775599
416 1662.05169592756
417 1693.2944317841
418 1725.12427662911
419 1757.55215644733
420 1790.58936046655
421 1824.24733797361
422 1858.53775338926
423 1893.47249031072
424 1929.06365562912
425 1965.32358372509
426 2002.26484074418
427 2039.90022895201
428 2078.24281818869
429 2117.30584360177
430 2157.10286856682
431 2197.64768504218
432 2238.95434413528
433 2281.03716097906
434 2323.91071967755
435 2367.58987839263
436 2412.08977445792
437 2457.42580844258
438 2503.61373391415
439 2550.66953681427
440 2598.60952468273
441 2647.45035104197
442 2697.20886451927
443 2747.90234732133
444 2799.54836765544
445 2852.1648238414
446 2905.76995050009
447 2960.38232489251
448 3016.02087334978
449 3072.7048778277
450 3130.45398258971
451 3189.28820101832
452 3249.22792255869
453 3310.29391977208
454 3372.50735553248
455 3435.88979037435
456 3500.4632606889
457 3566.25000417996
458 3633.27288968187
459 3701.55514446214
460 3771.12043226081
461 3841.99286143469
462 3914.19699336021
463 3987.75785091536
464 4062.70092716659
465 4139.05219418852
466 4216.83811208358
467 4296.08563813237
468 4376.82223616828
469 4459.0758860444
470 4542.8750933814
471 4628.24889942494
472 4715.22689108567
473 4803.83921125618
474 4894.11656916849
475 4986.09025112281
476 5079.79213128014
477 5175.25468270149
478 5272.51098866692
479 5371.59475405959
480 5472.54031707435
481 5575.38266115154
482 5680.15742705888
483 5786.90092591983
484 5895.65015337816
485 6006.44282838998
486 6119.31727754495
487 6234.31450341685
488 6351.47069014345
489 6470.82825780018
490 6592.42857061463
491 6716.31377003514
492 6842.52678939939
493 6971.1113687446
494 7102.11207001443
495 7235.5742924798
496 7371.54428840224
497 7510.06917921966
498 7651.19697171955
499 7794.97657478827
500 7941.45781625156
501 8090.69146027221
502 8242.72922484541
503 8397.62379973776
504 8555.42886476958
505 8716.19910850941
506 8879.99321342697
507 9046.86206851322
508 9216.86641223357
509 9390.05937199175
510 9566.51244611868
511 9746.28109791057
512 9929.42762751673
513 10116.0155057107
514 10306.109395893
515 10499.7751764932
516 10697.0799638079
517 10898.0921352467
518 11102.8813603283
519 11311.5186081129
520 11524.0761791553
521 11740.6277368844
522 11961.2483493735
523 12186.0144295259
524 12415.2039579105
525 12648.4999298552
526 12886.1795656281
527 13128.1138229272
528 13374.8054680036
529 13626.1324834874
530 13882.1819686205
531 14143.0426592537
532 14408.8049585952
533 14679.5609685381
534 14955.4045215803
535 15236.4534379337
536 15522.7610861668
537 15814.4484952263
538 16111.6167517313
539 16414.3688416242
540 16722.8096857476
541 17037.0461766933
542 17357.1872154991
543 17683.3437493578
544 18015.6288103693
545 18354.1575543151
546 18699.0473009229
547 19050.4175742938
548 19408.3901444398
549 19773.0890694725
550 20144.6407382831
551 20523.1739151567
552 20908.8197836477
553 21301.7119920326
554 21701.9867001327
555 22109.7826260715
556 22525.2410946099
557 22948.485543418
558 23379.6952523796
559 23819.0155629417
560 24266.5907639674
561 24722.5759660115
562 25187.1291941625
563 25660.4114428671
564 26142.5867316091
565 26633.822161877
566 27134.2879749703
567 27644.1576110855
568 28217.435182665
569 28747.6572907732
570 29231.9892319449
571 29781.2749648446
572 30340.8818536468
573 30911.0038342302
574 31491.8384867709
575 32083.5871036745
576 32687.0031201288
577 33301.2090564359
578 33926.9560124299
579 34564.4608414468
580 35213.9444733589
581 35875.6319938868
582 36549.5481514499
583 37236.3318600519
584 37936.0203324703
585 38648.8560497668
586 39375.0860504549
587 40114.9620151062
588 40868.7403546516
589 41636.6822956684
590 42419.0539744203
591 43216.126526325
592 44028.176184564
593 44855.4843707196
594 45698.3377934804
595 46557.0285522123
596 47431.8542308994
597 48323.1180082295
598 49231.1287570802
599 50156.201156159
};
\addlegendentry{Numerical lower bound vs numLayers}
\end{axis}

\end{tikzpicture}

%% file: AtBe25.bbl
\begin{thebibliography}{10}

\bibitem{BaCo17}
Heinz~H. Bauschke and Patrick~L. Combettes.
\newblock {\em Convex Analysis and Monotone Operator Theory in Hilbert Spaces}.
\newblock Springer International Publishing, 2017.

\bibitem{Be17}
Amir Beck.
\newblock {\em First-Order Methods in Optimization}.
\newblock Society for Industrial and Applied Mathematics, October 2017.

\bibitem{BeRaSc22}
Arash Behboodi, Holger Rauhut, and Ekkehard Schnoor.
\newblock {\em Compressive Sensing and Neural Networks from a Statistical
  Learning Perspective}, pages 247--277.
\newblock Springer International Publishing, 2022.

\bibitem{CaLiSo15}
Emmanuel~J. Candes, Xiaodong Li, and Mahdi Soltanolkotabi.
\newblock Phase retrieval via {W}irtinger flow: Theory and algorithms.
\newblock {\em {IEEE} Transactions on Information Theory}, 61(4):1985--2007,
  April 2015.

\bibitem{CoThOp96}
W.M.J. Coene, A.~Thust, M.~Op~de Beeck, and D.~Van~Dyck.
\newblock Maximum-likelihood method for focus-variation image reconstruction in
  high resolution transmission electron microscopy.
\newblock {\em Ultramicroscopy}, 64:109--135, 1996.

\bibitem{Co18}
Patrick~L. Combettes.
\newblock Monotone operator theory in convex optimization.
\newblock {\em Mathematical Programming}, 170(1):177--206, June 2018.

\bibitem{DaDeMo04}
Ingrid Daubechies, Michel Defrise, and Christine de~Mol.
\newblock An iterative thresholding algorithm for linear inverse problems with
  a sparsity constraint.
\newblock {\em Communications on Pure and Applied Mathematics},
  57(11):1413--1457, August 2004.

\bibitem{DoBe19}
Christian Doberstein and Benjamin Berkels.
\newblock A least-squares functional for joint exit wave reconstruction and
  image registration.
\newblock {\em Inverse Problems}, 35(5), 2019.

\bibitem{FoRa13}
Simon Foucart and Holger Rauhut.
\newblock {\em A Mathematical Introduction to Compressive Sensing}.
\newblock Springer New York, 2013.

\bibitem{GrLe10}
Karol Gregor and Yann LeCun.
\newblock Learning fast approximations of sparse coding.
\newblock In Johannes F{\"u}rnkranz and Thorsten Joachims, editors, {\em
  Proceedings of the 27th International Conference on Machine Learning
  (ICML-10)}, pages 399--406, Haifa, Israel, June 2010. Omnipress.

\bibitem{HeRoWe14}
John~R. Hershey, Jonathan~Le Roux, and Felix Weninger.
\newblock Deep unfolding: Model-based inspiration of novel deep architectures,
  September 2014.

\bibitem{HsChKa04}
Wen-Kuo Hsieh, Fu-Rong Chen, Ji-Jung Kai, and A.I Kirkland.
\newblock Resolution extension and exit wave reconstruction in complex hrem.
\newblock {\em Ultramicroscopy}, 98(2-4):99--114, January 2004.

\bibitem{LiToMo19}
Yuelong Li, Mohammad Tofighi, Vishal Monga, and Yonina~C. Eldar.
\newblock An algorithm unrolling approach to deep image deblurring.
\newblock In {\em ICASSP 2019 - 2019 IEEE International Conference on
  Acoustics, Speech and Signal Processing (ICASSP)}, pages 7675--7679. IEEE,
  May 2019.

\bibitem{MaYaKo21}
Yoshiki Masuyama, Kohei Yatabe, Yuma Koizumi, Yasuhiro Oikawa, and Noboru
  Harada.
\newblock Deep griffin--lim iteration: Trainable iterative phase reconstruction
  using neural network.
\newblock {\em IEEE Journal of Selected Topics in Signal Processing},
  15(1):37--50, January 2021.

\bibitem{MoLiEl21}
Vishal Monga, Yuelong Li, and Yonina~C. Eldar.
\newblock Algorithm unrolling: Interpretable, efficient deep learning for
  signal and image processing.
\newblock {\em {IEEE} Signal Processing Magazine}, 38(2):18--44, March 2021.

\bibitem{ShBe14}
Shai Shalev-Shwartz and Shai Ben-David.
\newblock {\em Understanding Machine Learning: From Theory to Algorithms}.
\newblock Cambridge University Press, May 2014.

\bibitem{Ta21}
Michel Talagrand.
\newblock {\em Upper and Lower Bounds for Stochastic Processes}.
\newblock Springer, 2nd edition, 2021.

\bibitem{XiWaGu21}
Zhuolei Xiao, Ya~Wang, and Guan Gui.
\newblock Smoothed amplitude flow-based phase retrieval algorithm.
\newblock {\em Journal of the Franklin Institute}, 358(14):7270--7285,
  September 2021.

\bibitem{ZhChWa24}
Xiaohan Zhang, Shaowen Chen, Shuya Wang, Ying Huang, Chuanhong Jin, and Fang
  Lin.
\newblock Exit wave reconstruction of a focal series of images with structural
  changes in high-resolution transmission electron microscopy.
\newblock {\em Journal of Microscopy}, 296(1):24--33, May 2024.

\end{thebibliography}
